\documentclass[final]{siamart}
\usepackage{docmute}

\usepackage{mathrsfs}
\usepackage{xcolor}

\usepackage{dsfont}
\usepackage{amsfonts}
\usepackage{amssymb}
\usepackage[a4paper, margin=30mm]{geometry}
\usepackage{thmtools,thm-restate}

\newcommand{\calF}{\mathcal{F}}

\newcommand{\calB}{\mathcal{B}}
\newcommand{\calN}{\mathcal{N}}

\newcommand{\calR}{\mathcal{R}}
\newcommand{\calO}{\mathcal{O}}
\newcommand{\calP}{\mathcal{P}}

\newcommand{\calC}{\mathcal{C}}

\newcommand{\calU}{\mathcal{U}}

\newcommand{\eps}{\epsilon}

\newcommand{\iid}{\overset{iid}{\sim}}

\newcommand{\law}{\mathrm{Law}}

\newcommand{\hess}{\mathrm{Hess}}
\newcommand{\lip}{\mathrm{Lip}}
\newcommand{\ul}[1]{\underline{#1}}
\newcommand{\ovl}[1]{\overline{#1}}

\newcommand{\wtilde}[1]{\widetilde{#1}}
\newcommand{\dd}{\mathrm{d}}
\newcommand{\st}{~|~}
\newcommand{\diff}[3][]{\frac{\dd^{#1} #2}{\dd {#3}^{#1}}}
\newcommand{\pdiff}[3][]{\frac{\partial^{#1} #2}{\partial {#3}^{#1}}}
\newcommand{\rl}[1]{\left\{\begin{array}{l l}#1\end{array}\right.}
\newcommand{\ip}[2]{\left\langle #1,#2\right\rangle}
\newcommand{\veps}{\varepsilon}

\newcommand{\R}{\mathbf{R}}
\renewcommand{\P}{\mathbf{P}}
\newcommand{\E}{\mathbf{E}}

\newcommand{\W}{\mathbf{W}}
\newcommand{\M}{\mathbf{M}}
\newcommand{\KL}{\mathbf{KL}}

\newcommand{\V}{\mathbf{V}}

\let\oldparagraph\paragraph
\renewcommand{\paragraph}[1]{\hfill\oldparagraph{#1}}

\renewtheorem{lemma}{Lemma}[section]
\renewtheorem{theorem}{Theorem}[section]
\renewtheorem{proposition}{Proposition}[section]
\renewtheorem{corollary}{Corollary}[section]
\newtheorem*{fact*}{Fact}
\newtheorem{remark}{Remark}[section]
\newtheorem{assumption}{Assumption}[section]

\usepackage{enumitem}

\newlist{assumptions}{enumerate}{1}

\setlist[assumptions]{
    leftmargin=2em,
    label=\textbf{\theassumprefix\arabic*.},
    ref=\theassumprefix\arabic*
}

\newcommand{\theassumprefix}{A}

\RenewDocumentEnvironment{assumptions}{o}
{
    \IfValueT{#1}{\renewcommand{\theassumprefix}{#1}}
    \begin{enumerate}[label=\textbf{\theassumprefix\arabic*.}, ref=\theassumprefix\arabic*]
}
{
    \end{enumerate}
}

\newcommand{\nnewline}{\newline}

\definecolor{linkblue}{HTML}{2263ab}

\hypersetup{
    colorlinks = true,
    linkbordercolor = {white},
    citecolor = {linkblue}
}

\makeatletter
\providecommand*\l@chapter       {\@dottedtocline{0}{0em}{1.5em}}
\providecommand*\l@section       {\@dottedtocline{1}{1.5em}{2.3em}}
\providecommand*\l@subsection    {\@dottedtocline{2}{3.8em}{3.2em}}
\providecommand*\l@subsubsection {\@dottedtocline{3}{7.0em}{4.1em}}
\providecommand*\l@paragraph     {\@dottedtocline{4}{10em}{5em}}
\providecommand*\l@subparagraph  {\@dottedtocline{5}{12em}{6em}}

\providecommand*\l@figure        {\@dottedtocline{1}{1.5em}{2.3em}}
\providecommand*\l@table         {\@dottedtocline{1}{1.5em}{2.3em}}

\newcommand\listoffigures{%
  \section*{\listfigurename}%
  \@starttoc{lof}%
}
\newcommand\listoftables{%
  \section*{\listtablename}%
  \@starttoc{lot}%
}
\makeatother

\usepackage{tocloft}

\renewcommand{\contentsname}{\centering\large\bfseries Contents}

\makeatletter
\newcommand{\tableofcontents}{
  \section*{\contentsname}
  \@starttoc{toc}
}
\makeatother

\title{Propagation of Chaos for Nonlinear Markov Chains}
\author{James Vuckovic (\url{james@jamesvuckovic.com})}

\begin{document}
	\maketitle
	
	\begin{abstract}
		We study 1-Wasserstein propagation of chaos for ``McKean-type'' nonlinear Markov chains and their associated interacting particle systems. This paper is organized into two parts: the first part combines arguments from various areas of nonlinear Markov theory into a systematic treatment of quantitative, nonasymptotic empirical measure estimates and propagation of chaos, with Lipschitz regularity as the primary tool. We also study extensions to uniform-in-time propagation of chaos and improved convergence rates under stronger assumptions such as transportation inequalities, modified metrics, or geometric ergodicity. The second part of this work consists of two detailed applications of our results to specific systems of interest: an Euler-Maruyama scheme for the standard McKean-Vlasov diffusion, and particle filtering via Feynman-Kac distribution flows.
	\end{abstract}

	\tableofcontents
	
    \section{Introduction}\label{sec:introduction}
    
    In this work, we study a general class of discrete-time ``McKean-type'' nonlinear Markov chains, by which we mean stochastic processes $\{X_n\}_{n=0}^\infty$ such that 
    \[
      \P(X_{n+1}\in \bullet|X_n,\dots, X_0) = K_{\eta_n}(X_n, \bullet);~~~\eta_n := \law(X_n)
    \]for some family of Markov kernels $\{K_\eta\st \eta\in\calP(\R^d)\}$. These processes generalize the notion of a Markov chain to the case where the evolution depends not just on $X_{n}$ but also the \emph{law} of $X_{n}$. We also consider the related system of particles $\ovl{X}^N_n:=(X^1_n, \dots, X^N_n)$, initially i.i.d., such that
    \[
      \P(X^i_{n+1}\in \bullet|\ovl{X}^N_{n}, \dots, \ovl{X}^N_{0}) = K_{m(\ovl{X}^N_{n})}(X^i_{n}, \bullet);~~~m(\ovl{X}^N_{n}):=\frac{1}{N}\sum_{i=1}^N \delta_{X^i_n}.
    \]In this setup, each particle $X^i_n$ has the same one-step transition as $X_n$ except that the law of $X_n$ is replaced with the empirical measure above. This models the popular family of interacting particle systems where the particles interact via their common empirical measure. Such systems arise in physics, such as in the modelling of granular media \cite{carrillo2003kinetic} or polymers \cite{freed1981polymers}, models of complex systems such as mean field games \cite{lasry2007mean,caines2006large}, and probabilistic algorithms such as Bayesian filtering \cite{handschin1970monte} or advanced Markov chain Monte Carlo (MCMC) techniques \cite{andrieu2007non,vuckovic2022nonlinear,clarte2022collective}.\nnewline
        
    Intuitively, if $N$ is large and $K_\bullet$ is reasonably well-behaved, we should hope for a law of large numbers, i.e. that $m(\ovl{X}^N_n) \approx \law(X_n)$. To formulate this precisely, one possibility is to study subsets of (dependent) particles $X^{N,q}_n:=(X^1_n,\dots,X^q_n)$ and measure the closeness of $\eta^{N,q}_n:=\law(X^{N,q}_n)$ to the law of $q$ i.i.d. particles according to $\eta_n$, denoted $\eta_n^{\otimes q}$. In other words, one can attempt to study the convergence, if any, for $\eta^{N,q}_n $ to $\eta_n^{\otimes q}$ as a function of $N,q$ for some appropriate notion of convergence on $\calP(\R^d)$. If $\eta^{N,q}_n\to \eta^{\otimes q}_n$, this is called \emph{propagation of chaos} \cite{kac1959probability} and any convergence as strong or stronger than weak convergence implies, among other things, the convergence of $m(\ovl{X}^N_n)\to \law(X_n)$ in law \cite[Proposition~2.2]{sznitman1991topics}.\nnewline

    While an intuitive property, obtaining rigorous propagation of chaos estimates can be challenging. In particular, the particles $\ovl{X}^N_n$ are \emph{a priori} statistically dependent, so the usual techniques for laws of large numbers do not automatically apply, and furthermore the dependence can evolve over time. We aim to address these difficulties by providing a systematic treatment of propagation of chaos for the above class of nonlinear Markov chains through the use of Lipschitz regularity of the associated nonlinear Markov kernels.

    \paragraph{Main Result} In the first half of this work, we show that if $\mu\mapsto \int \mu(\dd x) K_\mu(x, \dd y)$ is an \text{$L$-Lipschitz} mapping of $\calP_1(\R^d)\to\calP_1(\R^d)$ and if the $p$-th moments $\E[\|X^1_k\|^p]\leq M_k^{(p)}$ are bounded for some $p>1$ and $k=0,1,\dots,n$  then
    \[
        \W_1(\eta^{N,q}_n,\eta^{\otimes q}_n)\lesssim \frac{q}{N^{1/d}}\sum_{k=0}^n \left(M_k^{(p)}\right)^{1/p} L^{n-k} + \frac{q^3}{N}.
    \]We also show that the second term can be dropped under a slightly stronger type of Lipschitzness. The full statement is in Theorem~\ref{thm:poc_general}, where we show how these two ingredients --- regularity and moment control --- can be used to derive quantitative propagation of chaos rates, including uniform-in-time estimates under appropriate conditions, thereby sidestepping the issues described above for a broad class of systems.\nnewline

    \paragraph{Applications} In the second half of this work, we show how the general result described above may be applied to specific, yet representative, nonlinear Markov chains of interest. The first is an Euler-Maruyama discretization of the McKean-Vlasov SDE, see e.g.,  \cite{bossy1997stochastic}, for which we obtain a propagation of chaos result in the uniformly convex setting using mainly elementary arguments. The second application is propagation of chaos for nonlinear filtering in the context of Feynman-Kac distribution flows \cite{del2004feynman} where we show that, in particular, the update kernel for the popular SIS-R nonlinear filter \cite{gordon1993novel} is locally Lipschitz and derive propagation of chaos therefrom. The background for both of these applications is in Section~\ref{subsec:related_work}, where we argue that these represent two important classes of nonlinear Markovian systems under active study.

    \subsection{Related Work}\label{subsec:related_work}

    Though our results deal with discrete-time Markov chains, we first highlight some important related results in the continuous-time nonlinear Markov process literature as it has been the source of many technical developments and will be relevant to discussions in Sections~\ref{sec:propagation_of_chaos} and \ref{sec:mv}. This subject is vast so we highlight some important developments and their relevance to our work rather than attempt a comprehensive survey of the field.

    \paragraph{Continuous Time} Let us formally introduce the McKean-Vlasov equation \cite{mckean1966class}, which is an object of central importance in the continuous-time nonlinear diffusion literature. Consider the nonlocal partial differential equation
    \begin{equation}\label{eq:mv_pde}
        \pdiff{\eta_t}{t} = \Delta \eta_t + \nabla\cdot \left(\eta_t\nabla(V + W*\eta_t)\right)
    \end{equation}
    for some density $\eta_t$ on $\R^d$, where $V:\R^d\to\R$ is a confinement potential and $W:\R^d\to\R$ is an interaction potential\footnote{This is a specific choice of interaction mechanism which has be generalized considerably.}. Under mild conditions, it can be shown that the nonlinear stochastic differential equation (SDE)
    \begin{equation}\label{eq:mv_sde}
        \dd X_t = - \nabla V(X_t)\dd t - \nabla W* \eta_t(X_t) \dd t + \sqrt{2}\dd B_t;~~\eta_t:=\law(X_t)
    \end{equation}
    has a solution, the law of this solution has a density, and this density solves \eqref{eq:mv_pde} \cite{sznitman1991topics}. Just as in the discrete-time case, we can ``approximate'' the nonlinear SDE \eqref{eq:mv_sde} by a system of linear SDEs
    \begin{equation}\label{eq:i_sde}
        \dd X^i_t = - \nabla V(X^i_t)\dd t - \nabla W* \eta^N_t(X^i_t) \dd t + \sqrt{2}\dd B^i_t;~~\eta^N_t:=\frac{1}{N}\sum_{i=1}^N \delta_{X^i_t}.
    \end{equation}

    There has been substantial progress in studying propagation of chaos for \eqref{eq:i_sde}. A particularly fruitful approach introduced in \cite{sznitman1991topics} is to couple $X^i_t$ with a copy of $X_t$ via a shared noise process $\dd B^i_t$ to obtain $\W_2$-propagation of chaos results. This has been substantially extended over time using sophisticated ``reflection'' couplings \cite{eberle2016reflection} and modified semi-metrics to cover rather general cases of non-convex potentials \cite{durmus2020elementary,guillin2022convergence}. Similar to our results, \emph{a priori} moment control plays important roles in \cite{bolley2007quantitative} and \cite{durmus2020elementary}, which use ``synchronous'' and ``reflection'' couplings respectively, to obtain various types of $1$-Wasserstein propagation of chaos results. The former also uses results on convergence of general empirical measures, which is in the same spirit as our work, although the details very are different.
    \nnewline

    Another important approach to studying propagation of chaos for \eqref{eq:mv_sde} leverages the remarkable success of functional inequalities and entropy dissipation \cite{bakry1997sobolev} to study the long-time behaviour of \eqref{eq:mv_sde} \cite{carrillo2003kinetic,carrillo2006contractions}. In \cite{malrieu2003convergence}, log-Sobolev inequalities are used with probabilistic techniques to prove long-time convergence and propagation of chaos in the uniformly convex case and later extended to the non-convex case in \cite{cattiaux2008probabilistic}. Moment control again features prominently in these approaches. More recently, uniform-in-$N$ log-Sobolev inequalities have been used in \cite{guillin2021uniform,guillin2022uniform} to obtain various uniform-in-time propagation of chaos results in diverse settings.\nnewline

    Lastly, we mention two other approaches to quantitative propagation of chaos that are related to our analysis. The first is the family of techniques developed in \cite{mischler2013kac,hauray2014kac,mischler2015new,kolokoltsov2010nonlinear} which leverage analytical techniques on $\calP(\R^d)$ and various other abstractions to study propagation of chaos in a general setting. The second set of techniques is from \cite{jabin2018quantitative,lacker2023hierarchies,lacker2023sharp} and leverages relative entropy estimates and the BBGKY hierarchy and to obtain propagation of chaos, with a particular emphasis on sharpness and rates in $(q,N)$. In particular, under sufficient conditions, it is shown that $\W_2(\eta^{N,q}_t, \eta_t^{\otimes q})\lesssim q/N$ and that this is sharp for $\W_2$. 

    \paragraph{Discrete Time}
    Whereas theoretical considerations led to enormous progress in the continuous-time theory of nonlinear Markov processes, applications and the analysis of certain classes of algorithm have been key sources of inspiration for the development of the discrete-time theory. Among different types of algorithms, nonlinear (or Bayesian) filtering theory has been, by far, the most fecund problem for the advancement of the theory of nonlinear Markov chains. Interacting particle approximations of nonlinear filtering algorithms, typically called ``particle filters'' or ``Monte Carlo filters'' \cite{handschin1970monte}, are very practical and natural applications of nonlinear Markov theory.  Heuristic descriptions of these algorithms came first \cite{gordon1995bayesian,kitagawa1996monte,carvalho1997optimal,djuric2003particle}, but were soon followed with rigorous theory for large-particle convergence \cite{del1998measure}, central limit theorems \cite{del1999central}, long-time stability \cite{del2001stability}, and sharp propagation of chaos \cite{del2007sharp}.\nnewline

    Of particular relevance is the monograph \cite{del2004feynman}, which contains a remarkable consolidation of Bayesian filtering theory under the umbrella of Feynman-Kac distribution flows. This theory describes, among other things, a family of nonlinear Markov chains having selection-mutation dynamics modelled by a composition of a nonlinear \emph{selection} kernel $S_\eta$ and a \emph{mutation} kernel $M$
    \[
        K_\eta(x, \dd y) = \int S_\eta(x, \dd z) M( z, \dd y).
    \]Appropriate choices of $S_\bullet,M$ result in models and algorithms for, in particular, Kalman filters, SIS-R particle filters, genetic algorithms, and more. Results in \cite{del2004feynman} include pathwise and density-profile estimates, long-time stability, LLN and central limit theorems, and propagation of chaos, among other areas. The techniques used leverage algebraic ``stability'' properties of this family of nonlinear Markov kernels under composition and --- particularly relevant to our work --- semigroup Lipschitz regularity in various $h$-entropies including KL divergence and total variation. In particular, the main result of this work is directly inspired by \cite[Theorem~8.3.3]{del2004feynman} and can be seen as an extension to larger classes of nonlinear Markov chains and to the 1-Wasserstein distance. This is made explicit in Section~\ref{sec:fkm}, which applies Theorem~\ref{thm:poc_general} to Feynman-Kac models in the context of Bayesian filtering to obtain Theorem~\ref{thm:fk_poc}. \nnewline 

    Another important area of application for nonlinear Markov chains is advanced Markov chain Monte Carlo (MCMC) algorithms \cite{metropolis1953equation,metropolis1949monte}. The main idea is to construct a nonlinear Markov kernel satisfying $\pi(\dd y) = \int \pi(\dd x) K_\pi(x, \dd y)$ for some target distribution $\pi$ such that the interacting particle system converges rapidly to something approximating $\pi$. An ``interacting'' version of the popular Metropolis-Hastings algorithm \cite{hastings1970monte} already appeared in \cite{del2004feynman}, and related ideas were explored theoretically in \cite{andrieu2010particle,moral2010interacting} and also with applications to high-dimensional sampling in \cite{clarte2022collective} and Bayesian machine learning in \cite{vuckovic2022nonlinear}. In \cite{clarte2022collective}, an explicit coupling is used along with the LLN result in \cite{fournier2015rate} the to obtain 1-Wasserstein propagation of chaos; this is analogous to the continuous-time approach and leverages the special structure of their setup. In \cite{vuckovic2022nonlinear}, total variation propagation of chaos is obtained using a combination of standard Markov chain contraction arguments and LLN estimates from \cite{del2004feynman} for specific types of interaction mechanisms of interest. At the intersection of particle filtering and MCMC is sequential Monte Carlo filters \cite{del2006sequential,del2012adaptive,andrieu2010particle,rainforth2016interacting}, which propose a sequence of intermediate distributions and importance-sampling ideas for improved sampling from difficult target distributions or nonstationary systems, and have gained popularity for a variety of applications \cite{doucet2013sequential}.\nnewline

    There are also a number of works dealing with general nonlinear Markov kernels in a similar setting to ours. For example, \cite{doi:10.1137/S0040585X97986825} considers Dobrushin's condition for nonlinear Markov chains on general state spaces, and shows that a na\"ive application of the linear theory turns out to not be sufficient to guarantee the existence and uniqueness of an invariant measure; additional regularity of the transition kernel is necessary. These ideas are further developed in \cite{shchegolev2022new,shchegolev2022convergence,xu2022estimation}, and deal primarily with long-time questions of ergodicity for the mean-field system rather than particle approximations and their convergence.

    \subsection{Plan of the Paper}
    The rest of the paper is organized as follows: in Section~\ref{sec:propagation_of_chaos}, we derive the empirical measure and propagation of chaos results in terms of the regularity of $K_\bullet$ and describe various corollaries and implications, as well as a strategy for moment control. In Section~\ref{sec:uniform_propagation_of_chaos}, we explore sufficient conditions for uniform-in-time propagation of chaos based on the previous results. Finally, we apply these arguments to obtain quantitative propagation of chaos estimates for the two systems of significant interest described above: an Euler-Maruyama discretization of the McKean-Vlasov nonlinear diffusion in Section~\ref{sec:mv}, and nonlinear filtering in Section~\ref{sec:fkm}.

    \section{Propagation of Chaos}\label{sec:propagation_of_chaos}
    First we will introduce the notation and main constructions under consideration in this paper, then prove the main result for propagation of chaos, and then finally explore some corollaries and extensions.
    \subsection{Preliminaries}
    \paragraph{Notation} We deal with the Borel measurable space $(\R^d, \calB(\R^d))$, and the set of probability measures $\calP(\R^d)$ on this space. For $\mu\in \calP(\R^d)$, we denote the $p$-th moment of $\mu$ is
    \[
        \M_p(\mu):= \int \|x\|^p \mu(\dd x)
    \]and we will use the notation $\V_p(x):=\|x\|^p$ so in particular $\M_p(\mu)=\int \V_p \dd\mu=\mu(\V_p)$. We will also write $\calP_p(\R^d):=\{\mu\in\calP(\R^d)\st \M_p(\mu)<\infty\}$. The $p$-Wasserstein distance for $p\geq 1$ between $\mu,\nu\in\calP_p(\R^d)$ is defined as
    \[
        \W_p(\mu,\nu):=\inf_{\gamma\in\calC(\mu,\nu)}\left(\int \|x - y\|^p \gamma(\dd x, \dd y)\right)^{1/p}
    \]with $\calC(\mu,\nu)$ the set of couplings of $\mu,\nu$ and $\|\bullet\|$ the standard Euclidean metric. 
    On the $q$-fold tensor product space $(\R^d)^{\otimes q}$, we use the $\ell_2$ norm, i.e., which is the natural extension of the Euclidean norm
    \[
        \|(x^1, \dots, x^q) - (y^1, \dots, y^q)\| := \left(\sum_{i=1}^q \|x^i - y^i\|^2\right)^{1/2}.
    \]

    Let $K:\R^d\times \calB(\R^d)\to [0,1]$ be a Markov kernel and $f:\R^d\to\R$ be measurable; then we write $Kf(x):=\int K(x, \dd y)f(y)$. Let $\mu\in\calP(\R^d)$; then we write $\mu K(\dd y):=\int \mu(\dd x)K(x,\dd y)$. Finally let $K,M$ be Markov kernels; then $KM(x, \dd y):=\int K(x, \dd z)M(z, \dd y)$. For a collection of points $x^1,\dots,x^n\in \R^d$, we denote by $\ovl{x}:=(x^1,\dots,x^n)$ and $m(\ovl{x})(\dd x):=\frac{1}{N}\sum_{i=1}^n \delta_{x^i}(\dd x)\in \calP(\R^d)$.

    \paragraph{Basic Constructions} We will consider a family of Markov kernels indexed by probability measures, not necessarily time-homogeneous, denoted by $K^{(n)}_\bullet := \{K^{(n)}_\eta\st \eta\in \calP(\R^d)\}$ where each $K^{(n)}_\eta$ is a Markov kernel and $n\in \{1,2,3,\dots\}$. When this family is time-homogeneous, we will drop the superscript notation and just write $K_\bullet$. The $p$-Wasserstein contraction coefficient $\tau_p(K^{(n)}_\bullet)$ is defined as
    \[
        \tau_p(K^{(n)}_\bullet):=\sup_{\mu,\nu\in \calP_p(\R^d), ~\mu\neq \nu}\frac{\W_p(\mu K^{(n)}_\mu, \nu K^{(n)}_\nu)}{\W_p(\mu,\nu)}
    \]i.e., it is the best constant such that
    \[
        \W_p(\mu K^{(n)}_\mu, \nu K^{(n)}_\nu)\leq \tau_p(K^{(n)}_\bullet)\W_p(\mu,\nu).
    \]

    These ``nonlinear'' Markov kernels define nonlinear mappings $\Phi_n:\calP(\R^d)\to\calP(\R^d)$ by $\mu\mapsto \Phi_n[\mu]:=\mu K^{(n)}_\mu$. These nonlinear mappings form a nonlinear semigroup $\Phi_{m,n}$ defined for $n \geq m$ as
    \[
        \Phi_{m,n}:=\rl{
            \Phi_{n}\circ\cdots\circ\Phi_{m+1} & n > m\\
             \mathrm{Identity} & n = m.
        }
    \]Much of the semigroup-related notation in this work is derived from \cite{del2004feynman}. The kernels $K^{(n)}_\bullet$ can be used to define a time-inhomogeneous, nonlinear Markov chain: given an initial measure $\eta_0\in \calP(\R^d)$ and $X_0\sim \eta_0$, we can define 
    \begin{equation}\label{eq:mean_field}
        X_{n+1}\sim K^{(n+1)}_{\eta_n}(X_n, \bullet);~~~\eta_n:=\law(X_n).
    \end{equation}
    It is clear that that $\eta_n$ satisfies $\eta_{n+1}=\eta_n K^{(n+1)}_{\eta_n} = \Phi_{n+1}[\eta_n]$, and also that $\eta_n=\Phi_{0, n}[\eta_0]$. We will refer to $X_n$ or $\eta_n$ as the ``mean-field'' system. We will also use $\tau_1(\Phi_n)$ interchangeably with $\tau_1(K^{(n)}_\bullet)$.\nnewline

    Given an initial measure $\eta_0\in\calP(\R^d)$ and $X^1_0,\dots,X^N_0\iid \eta_0$ with $N\geq 1$, we can define the interacting particle system $\ovl{X}^N_n:=(X^1_n,\dots,X^N_n)$ associated with the mean field dynamics as 
    \begin{equation}\label{eq:ips}
        X^i_{n+1}\sim K^{(n+1)}_{\eta^N_n}(X^i_n,\bullet);~~~\eta^N_n:=m(\ovl{X}^N_n).
    \end{equation}
    The following important property is an obvious consequence of the symmetry and transition mechanisms for the system above, and will be used throughout this work.
    \begin{fact*}
        When $X^1_0, \dots, X^N_0$ are exchangeable, the system $(X^1_n,\dots,X^N_n)$ is exchangeable, i.e., its joint law is invariant under permutation. However, unlike the mean-field case, it is \emph{not} true that $\eta^N_{n+1}=\Phi_{n+1}[\eta^N_n]$, but rather $(X^1_{n+1}, \dots, X^N_{n+1})$ are \emph{conditionally i.i.d.} given $\calF^N_{n}:= \sigma(X^1_n, \dots, X^N_n)$ with common law $\Phi_{n+1}[\eta^N_n]$.
    \end{fact*}

    We are interested in the joint law of $q$-subsets of $(X^1_n,\dots,X^N_n)$; by excheageability it is sufficient to define $\eta^{N,q}_n:=\law(X^1_n, \dots, X^q_n)$ for $1\leq q\leq N$. In particular, we are interested in the extent to which $\eta^{N,q}_n\approx \eta_n^{\otimes q}$.

    \subsection{Main Result}
    The results below establish quantitative propagation using three basic ingredients: Lipschitzness of the map $\eta\mapsto \Phi_n[\eta]$; moment control through time of $\E[\|X^1_n\|^p]$, and the follwing sharp and general estimate for Wasserstein distances between empirical measures and their mean-field counterparts. Let $\mu^N$ be an empirical measure over $N$ i.i.d. samples from $\mu\in\calP(\R^d)$. Then \cite[Theorem~1]{fournier2015rate} says that, for $r\in(0,\infty)$, $p>r$ and $\M_p(\mu)<\infty$, we have
    \begin{equation}\label{eq:fournier}
        \E_{\mu}[\W^r_r(\mu^N, \mu)]\leq C_{r, p, d}\M_p(\mu)^{r/p}\calR_{r, p, d}(N)
    \end{equation}
    where $\calR_{r, p, d}$ is the rate function
    \begin{equation}
        \calR_{r,p,d}(N):=\rl{
                N^{-1/2} + N^{-(p-r)/p}& \text{ if }r>d/2,~p\neq 2p\\
                N^{-1/2}\log(1 + N) + N^{-(p-r)/p}& \text{ if }r=d/2,~p\neq 2p\\
                N^{-r/d} + N^{-(p- r)/p} & \text{ if }r <d/2,~p\neq d/(d-r)
            }
    \end{equation}
    and $C_{r,p,d}$ is a constant depending only on $r,p,d$. This estimate is sharp up to constants, which are not specified here but are studied in detail in \cite{fournier2023convergence}.\nnewline

    We will impose some general assumptions in Assumption~\ref{assump:global} throughout the paper to improve clarity and sidestep technical issues. We will assume they hold in the sequel unless stated otherwise.
    \begin{assumption}[General Assumptions]\label{assump:global}\hfill
        \begin{assumptions}
            \item We assume that $d\geq 3$ and, whenever we are applying \cite{fournier2023convergence} we have $p\neq d/(d-1)$ unless otherwise stated, where $p$ is the number of moments in the statement above.
            \item For $K^{(n)}_\bullet$ as defined above, we have for each $\eta\in\calP_1(\R^d)$
            \[
                \int \|y\| K^{(n)}_\eta(x,\dd y) \leq C_0 + C_1\|x\|~~~\forall\|x\|>R
            \]for some $C_0,C_1\geq 0$ possibly depending on $\eta$, and some $R\geq 0$.
        \end{assumptions}
    \end{assumption}
     The first assumption will allow us to improve notational clarity by avoiding edge cases in \cite[Theorem~1]{fournier2015rate} since we will deal only with the case of $r=1$ in this work. In this case, we may write
    \begin{equation}
        \E_\mu[\W_1(\mu^N, \mu)]\leq C_{p,d}\M_p(\mu)^{1/p}\frac{1}{N^{1/d}}.
    \end{equation}
    Most of our results hold without the first assumption by replacing the estimate above with \eqref{eq:fournier}. The second assumption ensures that $\mu K_\eta\in \calP_1(\R^d)$ when $\mu,\eta\in\calP_1(\R^d)$ and allows us to give meaning to $\W_1$-Lipschitz estimates for $K_\bullet$. This is also not burdensome as we will obtain \emph{a priori} $p$-th moment bounds for all kernels of interest.\nnewline

    Now we enumerate the various types of Lipchitz regularity studied in the remainder of this paper. We will state which ones are assumed in each result.
    \begin{assumption}[Lipschitz Regularity]
        Let $K^{(n)}_\bullet$ be as above. 
        \begin{assumptions}[R]
            \item \label{assump:R1} $K^{(n)}_\bullet$ is a Lipschitz-continuous family of nonlinear Markov kernels, i.e., there are constants $\tau_1(K^{(n)}_\bullet)\in (0, \infty)$ such that for every $\mu,\nu\in \calP_1(\R^d)$
            \[
                \W_1(\mu K^{(n)}_\mu, \nu K^{(n)}_\nu)\leq \tau_1(K^{(n)}_\bullet)\W_1(\mu,\nu).
            \]
            \item \label{assump:R2} There are constants $c,C\in (0,\infty)$ such that, for every $x,y\in \R^d$ and $\mu,\nu\in\calP_1(\R^d)$, we have
            \[
                \W_1(K^{(n)}_\mu(x,\bullet), K^{(n)}_\nu(y,\bullet))\leq c\|x - y\| + C\W_1(\mu,\nu).
            \]
            \item \label{assump:R3} $K^{(n)}_\bullet$ satisfies the following local one-sided Lipschitz condition: for every $\mu,\nu\in \calP_1(\R^d)$
            \[
                \W_1(\mu K^{(n)}_\mu, \nu K^{(n)}_\nu)\leq \tau_1(K^{(n)}_\bullet)[\nu]\cdot \W_1(\mu,\nu)
            \]where now $\tau_1(K^{(n)}_\bullet)[\nu]\in (0,\infty)$ depends on $\nu\in\calP_1(\R^d)$.
        \end{assumptions}
    \end{assumption}
    We will see below that Assumption~\ref{assump:R2} implies Assumption~\ref{assump:R1}. Also, by symmetry, Assumption~\ref{assump:R3} applies equally to either $\mu$ or $\nu$ in the context above, or it can be symmetrized. Now we also establish a simple consequence of exchangeability that we will use throughout the paper.\nnewline

    \begin{lemma}\label{lem:emp_mom}
        Suppose that $(X^1_n, \dots, X^N_n)$ follows the interacting particle dynamics \eqref{eq:ips}, and $\eta^N_n$ is the empirical measure over the $N$ particles. Let $p\in [1,\infty)$; we have
        \[
            \E[\|X^1_n\|^p]= \E[\M_p(\eta^N_n)] = \E[\M_p(\Phi_n[\eta^N_{n-1}])]
        \]where the equalities are understood in $[0,\infty]$.
    \end{lemma}
    \begin{proof}
        On one hand by exchangeability
        \begin{align*}
            \E[\M_p(\eta^N_n)] = \frac{1}{N}\sum_{i=1}^N \E[\|X^i_n\|^p] = \E[\|X^1_n\|^p].
        \end{align*}
        On the other, using $\calF^N_n:=\sigma(X^1_n,\dots,X^N_n)$ and $X^i_n \sim K^{(n)}_{\eta^N_{n-1}}(X^i_{n-1},\bullet)$, we get
        \begin{align*}
            \E[\|X^1_n\|^p]&=\frac{1}{N}\sum_{i=1}^n \E[\|X^i_n\|^p]\\
            &=\frac{1}{N}\sum_{i=1}^n \E[\E[\|X^i_n\|^p|\calF^N_{n-1}]]\\
            &= \frac{1}{N}\sum_{i=1}^N \E[K^{(n)}_{\eta^N_{n-1}}\V_p(X^i_{n-1})]\\
            &= \E\left[\eta^N_{n-1} K^{(n)}_{\eta^N_{n-1}}(\V_p)\right]\\
            &= \E[\Phi_n[\eta^N_{n-1}](\V_p)].
        \end{align*}
    \end{proof}

    \paragraph{A Quantitative Empirical Measure Estimate}
    We are ready to prove the main results, which are quantitative empirical measure estimates for \eqref{eq:ips} that directly result in propagation of chaos; the first result is Proposition~\ref{prop:emp_main}. The proof is not particularly complicated and essentially generalizes the argument from the proof of \cite[Theorem~8.3.3]{del2004feynman} by leveraging the favourable properties of Wasserstein distances for empirical measures and estimates from \cite{fournier2015rate}.

    \begin{proposition}\label{prop:emp_main}
        Suppose that Assumption~\ref{assump:R1} holds, i.e., there is $\tau_1(K^{(n)}_\bullet)\in(0,\infty)$ with
        \[
            \W_1(\mu K^{(n)}_\mu, \nu K^{(n)}_\nu)\leq \tau_1(K^{(n)}_\bullet)\W_1(\mu,\nu) ~\forall \mu,\nu\in \calP_1(\R^d).
        \]Suppose that $\eta_n$ follows the mean-field dynamics~\eqref{eq:mean_field} and $(X^1_n,\dots, X^N_n)$ follows the associated particle dynamics~\eqref{eq:ips}. Finally, assume there exist constants $M^{(p)}_k >0$ constants satisfying
        \[
            \E[\|X^1_k\|^p]\leq M^{(p)}_k
        \]
        for some $p>1$ and $k=0,\dots,n$. Then for every $n=1,2,\dots,$ we have
        \[
            \E[\W_1(\eta^{N}_n, \eta_n)]\leq C_{p, d}\cdot \frac{1}{N^{1/d}}\sum_{k=0}^n \left(M^{(p)}_k\right)^{1/p}\tau_1(\Phi_{k,n}).
        \]
    \end{proposition}
    \begin{proof}
        Let $f\in \lip_1(\R^d)$ be arbitrary and consider the decomposition, which comes from \cite[(8.23)]{del2004feynman},
        \begin{align*}
            \eta^N_n(f) - \eta_n(f) =& \Phi_{n,n}[\eta^N_n](f) - \Phi_{0,n}[\eta_0](f)\\
            =& \Phi_{n,n}[\eta^N_n](f) - \Phi_{n,n}[\Phi_{n-1,n}[\eta^N_{n-1}]](f)\\
            &+ \Phi_{n,n}[\Phi_{n-1,n}[\eta^N_{n-1}]](f) - \Phi_{n-1,n}[\Phi_{n-2,n-1}[\eta^N_{n-2}]](f)+\cdots\\
            &+ \Phi_{1,n}[\Phi_{0,1}[\eta^N_0]](f) - \Phi_{0,n}[\eta^N_{0}](f)\\
            &+\Phi_{0,n}[\eta^N_{0}](f) - \Phi_{0,n}[\eta_{0}](f)\\
            =&\sum_{k=0}^n \Phi_{k,n}[\eta^N_k](f) - \Phi_{k,n}[\Phi_{k-1, k}[\eta^N_{k-1}]](f)
        \end{align*}
        where $\Phi_{-1, 0}[\eta^N_{-1}]:=\eta_0$. Taking the supremum over all such $f$ and using Kantorovich duality gives 
        \[
            \W_1(\eta^N_n, \eta_n)\leq \sum_{k=0}^n \W_1(\Phi_{k,n}[\eta^N_k], \Phi_{k,n}[\Phi_k[\eta^{N}_{k-1}]]).
        \]
        As noted, $X^i_n$ are conditionally i.i.d. given $\calF^N_{n-1}:=\sigma(X^1_{n-1}, \dots, X^N_{n-1})$ with common law $\Phi_n[\eta^N_{n-1}]$. Thus taking expectations and using Lipschitz regularity of $\Phi$
        \begin{align*}
            \E[\W_1(\eta^N_n, \eta_n)]&\leq \sum_{k=0}^n \E[\W_1(\Phi_{k,n}[\eta^N_k], \Phi_{k,n}[\Phi_k[\eta^{N}_{k-1}]])]\\
            &\leq\sum_{k=0}^n \tau_1(\Phi_{k,n})\E[\W_1(\eta^N_k, \Phi_k[\eta^N_{k-1}])]\\
            &=\sum_{k=0}^n  \tau_1(\Phi_{k,n})\E[\E[\W_1(\eta^N_k, \Phi_k[\eta^N_{k-1}])|\calF^N_{k-1}]]\\
            &\leq\sum_{k=0}^n \tau_1(\Phi_{k,n})C_{p, d}\cdot \E[\M_p(\Phi_k[\eta^N_{k-1}])^{1/p}]\frac{1}{N^{1/d}}
        \end{align*}
        using \cite[Theorem~1]{fournier2015rate}. 
        Using $\E\left[\M_p(\Phi_k[\eta^N_{k-1}])^{1/p}\right]\leq \left(M^{(p)}_k\right)^{1/p}$ by Jensen's inequality and Lemma~\ref{lem:emp_mom}, we get
        \[
            \E[\W_1(\eta^{N}_n, \eta_n)]\leq C_{p,d}\cdot \frac{1}{N^{1/d}}\sum_{k=0}^n \left(M^{(p)}_k\right)^{1/p}\tau_1(\Phi_{k,n}).
        \]
    \end{proof}

    These ingredients are by now fairly standard for propagation of chaos; various forms of regularity and moment control appear in many papers studying propagation of chaos in discrete and continuous time. However, our result leverages these elements in a general setting without the need for bespoke coupling or combinatorial arguments, which allows for broad applicability. Moreover, as we will see in Sections~\ref{sec:mv} \& \ref{sec:fkm}, Lipschitz regularity can often be determined \emph{algebraically} from the structure of the kernels $K^{(n)}_\bullet$; therefore this approach transforms a \emph{probabilistic, dyanmic} question into an \emph{algebraic} one which is often much easier to answer.\nnewline

    Now we prove a similar result that holds in the more general setting of a one-sided local Lipschitz estimate. This generalization will become very useful when studying kernels with jumps, e.g., in Section~\ref{sec:fkm}. 
    \begin{proposition}\label{prop:emp_main_oneside}
        Suppose that Assumption~\ref{assump:R3} holds, i.e., for every $\mu,\nu\in \calP_1(\R^d)$ we have the one-sided local Lipschitz estimate
        \[
            \W_1(\mu K^{(n)}_\mu, \nu K^{(n)}_\nu)\leq \tau_1(K^{(n)}_\bullet)[\nu]\cdot \W_1(\mu,\nu),~\tau_1(K^{(n)}_\bullet)[\nu]\in(0,\infty).
        \]Suppose that $\eta_n$ follows the mean-field dynamics~\eqref{eq:mean_field} and $(X^1_n,\dots, X^N_n)$ follows the associated particle dynamics~\eqref{eq:ips}. Finally, assume there exist constants $M^{(p)}_k >0$ satisfying
        \[
            \E[\|X^1_k\|^p]\leq M^{(p)}_k
        \]
        for some $p>1$ and $k=0,\dots,n$. Then for every $n=1,2,\dots,$ we have
        \[
            \E[\W_1(\eta^{N}_n, \eta_n)]\leq C_{p, d}\frac{1}{N^{1/d}}\sum_{k=0}^n  \left(M^{(p)}_k\right)^{1/p}\prod_{j=k+1}^n \tau_1(\Phi_j)[\eta_{j-1}].
        \]recalling that $\Phi_j[\mu]:=\mu K^{(j)}_\mu$.
    \end{proposition}
    \begin{proof}
        Since we know that each $X^i_n$ is conditionally i.i.d. given $\calF^N_{n-1}:=\sigma(X^1_{n-1}, \dots,X^N_{n-1})$ with common law $\Phi_n[\eta^N_{n-1}]$, we can use this as a centering term to obtain the fundamental recursion
        \begin{align*}
            \E[\W_1(\eta^{N}_n, \eta_n)]&\leq \E[\W_1(\eta^{N}_n, \Phi_n[\eta^N_{n-1}])] + \E[\W_1(\Phi_n[\eta^N_{n-1}], \eta_n)]\\
            &\leq \E[\W_1(\eta^{N}_n, \Phi_n[\eta^N_{n-1}])] + \tau_1(\Phi_n)[\eta_{n-1}]\cdot \E[\W_1(\eta^N_{n-1}, \eta_{n-1})]\\
            &=\E[\E[\W_1(\eta^{N}_n, \Phi_n[\eta^N_{n-1}])|\calF^N_{n-1}]] + \tau_1(\Phi_n)[\eta_{n-1}]\cdot \E[\W_1(\eta^N_{n-1}, \eta_{n-1})]\\
            &\leq C_{p,d}\frac{1}{N^{1/d}}\E[\M_p(\Phi_n[\eta^N_{n-1}])^{1/p}] + \tau_1(\Phi_n)[\eta_{n-1}]\cdot \E[\W_1(\eta^N_{n-1}, \eta_{n-1})]
        \end{align*}
        by \cite[Theorem~1]{fournier2015rate}. Using Jensen's inequality to get $\E[\M_p(\Phi_n[\eta^N_{n-1}])^{1/p}]\leq (M^{(p)}_n)^{1/p}$ gives
        \[
            \E[\W_1(\eta^{N}_n, \eta_n)]\leq C_{p,d}\frac{1}{N^{1/d}}(M^{(p)}_n)^{1/p} + \tau_1(\Phi_n)[\eta_{n-1}]\cdot \E[\W_1(\eta^N_{n-1}, \eta_{n-1})]
        \]and unrolling this sum yields the result.
    \end{proof}
    \begin{remark}
        This result uses a simpler decomposition than Proposition~\ref{prop:emp_main} but is slightly weaker since in general
        \[
            \tau_1(\Phi_{k,n})\leq \prod_{j=k+1}^n \tau_1(\Phi_j).
        \]Also, note that the one-sidedness of the estimate is a useful feature since it allows us to estimate the regularity by a deterministic constant $\tau(\Phi_n)[\eta_{n-1}]$ rather than a random one $\tau(\Phi_n)[\eta^N_{n-1}]$, which is important for extracting a simple recursion. Furthermore, the mean field system $\eta_n$ is often much easier to work with, and indeed we can see that $\tau_1(\Phi_n)[\eta_{n-1}]=\tau_1(\Phi_n)[\Phi_{0, n-1}[\eta_0]]$ has a closed form expression in terms of the initial conditions and the kernels $K^{(k)}_\bullet$.
    \end{remark}

    \paragraph{Propagation of Chaos}\label{subsec:poc}
    The first propagation of chaos result holds in the general setting of Assumption~\ref{assump:R1} (and also Assumption~\ref{assump:R3} although for simplicity we only prove it for the former). The result uses the following tensorization lemma, the proof of which we defer to Appendix~\ref{app:auxiliary_results} as it requires some combinatorial machinery that is not used anywhere else in this work.
    \begin{restatable}{lemma}{TensorizationLemma}\label{lem:tensorization}
        Let $X:=(X^1,\dots,X^N)$ be exchangeable random variables with $\law(X)\in\calP_1((\R^d)^{\otimes N})$ and let $\nu\in\calP_1(\R^d)$. For any $q\in \{1,\dots,N\}$, let $\mu^{N,q}:=\law(X^1,\dots,X^q)$; then
        \[
            \W_1(\mu^{N,q},\nu^{\otimes q})\leq q\E[\W_1(m(X), \nu)] + 2\frac{q^3}{N}\M_1\left(\mu^{N,1}\right).
        \]
    \end{restatable}
    \begin{theorem}\label{thm:poc_general}
        Suppose that Proposition~\ref{prop:emp_main} is satisfied. Then for $q\in \{1,\dots,N\}$
        \[
            \W_1(\eta^{N,q}_n,\eta^{\otimes q}_n)\leq  C_{p, d}\frac{q}{N^{1/d}}\sum_{k=0}^n \left(M^{(p)}_k\right)^{1/p}\tau_1(\Phi_{k,n})  + 2\frac{q^3}{N}(M^{(p)}_n)^{1/p}.
        \]
    \end{theorem}
    \begin{proof}
        This is an immediate consequence of Lemma~\ref{lem:tensorization}, plus the fact that $\M_1(\eta^{N,1}_n)\leq \M_p(\eta^{N,1}_n)^{1/p}$.
    \end{proof}

    Theorem~\ref{thm:poc_general} holds in a general setting, but implies a dependence on $q$ which is\footnote{For simplicity, throughout the text we will use $f(n)=\calO(g(n))$ to mean ``$f(n)\leq C g(n)~\forall n$ for some $C>0$'', i.e., \emph{not} in the asymptotic sense.} $\calO(q^3)$. Under slightly stronger regularity assumptions, we can improve this to $\calO(q)$ by using the Markov chain equivalent of the classical synchronous coupling argument from \cite{sznitman1991topics}.

    \begin{lemma}\label{lem:tau}
        Suppose that $K_\bullet$ satisfies Assumption~\ref{assump:R2}. Then $\tau_1(K_\bullet) \leq c + C$.
    \end{lemma}
    \begin{proof}
        \begin{align*}
            \frac{\W_1(\mu K_\mu,\nu K_\nu)}{\W_1(\mu,\nu)}&\leq \frac{\W_1(\mu K_\mu, \mu K_\nu)}{\W_1(\mu,\nu)} + \frac{\W_1(\mu K_\nu, \nu K_\nu)}{\W_1(\mu,\nu)}\\
            &\leq \underbrace{\int \frac{\W_1(K_\mu(x, \bullet), K_\nu(x, \bullet))}{\W_1(\mu,\nu)}\mu(\dd x)}_{(1)} + \underbrace{\frac{\W_1(\mu K_\nu, \nu K_\nu)}{\W_1(\mu,\nu)}}_{(2)}
        \end{align*}
        $(1)\leq C$ since term with $x-y$ vanishes. For the second term, we use Proposition~\ref{prop:wass_contract} to get
        \begin{align*}
            (2)&\leq\sup_{x\neq y}\frac{\W_1(K_\nu(x, \bullet), K_\nu(y, \bullet))}{\|x - y\|}\leq c.
        \end{align*}
    \end{proof}

    For the sake of simplicity, we prove the result below for time-homogeneous kernels $K^{(n)}_\bullet = K_\bullet$, however the analogous result would hold in the time-inhomogeneous case as well.
    \begin{theorem}\label{thm:poc_good}
        Suppose that Proposition~\ref{prop:emp_main} is satisfied with $K_\bullet$ and that, in particular, $K_\bullet$ satisfies Assumption~\ref{assump:R2} with constants $c,C$. Then for every $n=1,2,\dots,$ and $q\in \{1.\dots, N\}$ we have
        \[
            \W_1(\eta^{N,q}_n, \eta_n^{\otimes q})\leq  C'_{p, d}\cdot \frac{q}{N^{1/d}}\sum_{k=0}^{n-1}c^{n-k-1}\sum_{j=0}^k \left(M^{(p)}_j\right)^{1/p}(c+C)^{k-j}.
        \]
    \end{theorem}
    \begin{proof}
        First, note that, under these assumptions and Lemma~\ref{lem:tau}, $\tau_1(K_\bullet)\leq c + C$, and Proposition~\ref{prop:emp_main} applies so that we have control on $\E[\W_1(\eta^N_n, \eta_n)]$ as above. We will combine this with a standard coupling approach to leverage this into a propagation of chaos estimate.\nnewline
        
        The Lipschitz assumption \ref{assump:R2} implies that there is a Markov coupling $\gamma_{x,y,\mu,\nu}\in\calC(K_\mu(x,\bullet), K_\nu(y,\bullet))$ such that
        \[
            \E_{W,Z\sim \gamma_{x,y,\mu,\nu}}[\|W - Z\|]= \W_1(K_\mu(x,\bullet), K_\nu(y,\bullet))\leq c\|x-y\| + C\W_1(\mu,\nu).
        \] We can use this to construct a coupling of $X^i_n$ and $Y^i_n\iid \eta_n$ for a given $i\in \{1,\dots,q\}$ as follows: let $(X^1_n,\dots,X^N_n)$ be the particles of \eqref{eq:ips} at time $n$ and to each $X^i_n$, we independently couple $Y^i_n$ by setting $X^i_0=Y^i_0$ (by assumption, their initial distributions are the same), and for $n>0$, we sample
        \[
            (X^i_n, Y^i_n)\sim \gamma_{X^i_{n-1}, Y^i_{n-1},\eta^N_{n-1},\eta_{n-1},}.
        \]By construction, $X^i_n\sim K_{\eta^N_{n-1}}(X^i_{n-1}, \bullet)$ and $Y^i_n\sim K_{\eta_{n-1}}(Y^i_{n-1},\bullet)$ so the marginals of this coupling follow the correct dynamics. For any $i=1,\dots, N$ we have
        \begin{align*}
            \E[\|X^i_n - Y^i_n\|] &\leq c\E[\|X^i_{n-1} - Y^i_{n-1}\|] + C\E[\W_1(\eta^N_{n-1}, \eta_{n-1})]\\
            &\leq c(c\E[\|X^i_{n-2} - Y^i_{n-2}\|] + C\E[\W_1(\eta^N_{n-2}, \eta_{n-2})]) + C\E[\W_1(\eta^N_{n-1}, \eta_{n-1})]\\
            &= c^2 \E[\|X^i_{n-2} - Y^i_{n-2}\|] + C(\E[\W_1(\eta^N_{n-1}, \eta_{n-1})] + c\E[\W_1(\eta^N_{n-2}, \eta_{n-2})])\\
            &\cdots\\
            &\leq c^n\E[\|X^i_0 - Y^i_0\|] + C\sum_{k=0}^{n-1} c^{n-k-1}\E[\W_1(\eta^N_k, \eta_k)]\\
            &=C\sum_{k=0}^{n-1} c^{n-k-1}\E[\W_1(\eta^N_k, \eta_k)].
        \end{align*}
        Therefore by exchangeability
        \begin{align*}
            \W_1(\eta^{N,q}_n, \eta^{\otimes q}_n)&\leq \sum_{i=1}^q \E[\|X^i_n - Y^i_n\|]\\
            &= q\E[\|X^1_n - Y^1_n\|]\\
            &\leq q C\sum_{k=0}^{n-1} c^{n-k-1}\E[\W_1(\eta^N_k, \eta_k)]
        \end{align*}
        and then plugging in
        \[
            \E[\W_1(\eta^{N}_k, \eta_k)]\leq C_{p, d}\frac{1}{N^{1/d}}\sum_{j=0}^k \left(M^{(p)}_j\right)^{1/p}(C+c)^{k-j}
        \]from Proposition~\ref{prop:emp_main}, we have
        \begin{align*}
            \W_1(\eta^{N,q}_n, \eta^{\otimes q}_n)&\leq C'_{p,d}\cdot \frac{q}{N^{1/d}}\sum_{k=0}^{n-1}c^{n-k-1}\sum_{j=0}^k\left(M^{(p)}_j\right)^{1/p}(C + c)^{k-j}
        \end{align*}
        as desired.
    \end{proof}
    \begin{remark}
        Note that, assuming $M^{(p)}_k\leq M^{(p)}_\star$, after some algebra this is still qualitatively ``geometric'' $\calO((C + c)^n)$ like the previous result.
    \end{remark}

    \subsection{Sharpness and Rates}\label{subsec:rates}

    Let us first note that Proposition~\ref{prop:emp_main} is sharp at the current level of generality, i.e., it is not hard to construct nontrivial, contractive nonlinear Markov kernels that attain the $\E[\W_1(\eta^N_n,\eta_n)]\gtrsim N^{-1/d}$ bound using arguments from \cite{fournier2015rate}. Indeed, the following simple nonlinear Markov kernel is instructive.
    \begin{proposition}\label{prop:sharp}
        Let $\calU_d$ be the uniform distribution on $[-1,1]^d$ and define
        \[
            K_\eta(x,\dd y) = \frac{1}{3}\delta_x(\dd y) + \frac{1}{3}\eta(\dd y) + \frac{1}{3}\calU_d.
        \]Then $K_\bullet$ is a strict contraction, and there is $\eta_0$ such that $\E[\|X^1_n\|^p]\leq M^{(p)}_\star<\infty~\forall p>1,~n\geq 0$ and 
        \[
            \E[\W_1(\eta^{N}_n, \eta_n)]\gtrsim N^{-1/d}
        \]uniformly in $n$.
    \end{proposition}
    \begin{proof}
        For the same of simplicity, take $\eta_0$ to be bounded and supported on the unit cube; in particular, it has finite moments of all orders. The same is true of $X^1_n$ under these conditions. Moreover, $K_\bullet$ is a strict contraction such that
        \[
            \W_1(K_\mu(x,\bullet), K_\nu(y,\bullet))\leq \frac{1}{3}\|x-y\| + \frac{1}{3}\W_1(\mu,\nu).
        \]

        Note that $\eta_n$ has a strictly positive density on the unit cube for $n\geq 1$ for any initial measure. Indeed,
        \[
            \eta_n=\left(\frac{2}{3}\right)^n \eta_0 + \left(1-\left(\frac{2}{3}\right)^n\right)\calU_d
        \]so that the density of $\eta_n$, also denoted $\eta_n$, satisfies $\eta_n(x)\geq (1-(2/3)^n)2^{-d}\geq (1-2/3)2^{-d}=3^{-1}2^{-d}$. Thus, classical results form quantization theory, e.g.,  \cite[Theorem 6.2]{graf2000foundations}, imply that
        \[
            \W_1(\eta^N_n, \eta_n)\gtrsim N^{-1/d}~~~\eta_n-a.s
        \]regardless of the dependence structure of $\eta^N_n$, or even if $\eta^N_n$ is not supported on $\calU_d$. Furthermore, since the cited quantization result only depends on the volume of the unit cube, the lower bound on the density, and the dimension, the above result holds \emph{uniformly} in $n$.
    \end{proof}

    This example reveals an important limitation of our technique: by leveraging empirical measure theory to establish propagation of chaos in the ways above, we cannot hope to overcome the curse of dimensionality, i.e., the $1/d$ in the exponent of $N$, at least in general. Indeed, the overall argument above holds for any measure having a density w.r.t. the Lebesgue measure on $\R^d$ and satisfying mild moment conditions (although perhaps not uniformly), and this result is a geometric fact of approximating continuous measures by discrete ones, not a probabilistic result. Furthermore, $\W_1(\eta^{N,q}_n,\eta_n^{\otimes q})$ may be \emph{much} better that $\calO(N^{-1/d})$. In fact, in the example above, it is easy to see that $\calU_d$ is invariant for $K_{\bullet}$, and if $\eta_0=\calU_d$ in Proposition~\ref{prop:sharp}, then $\eta^{N,1}_n=\calU_d~\forall n\geq 0$. This implies $\W_1(\eta^{N,1}_n, \eta_n)=\W_1(\eta^{N,1}_n, \calU_d)=0$, even though the same is not necessarily true for $q>1$ due to correlations between particles.\nnewline

    On the other hand, the results of Section~\ref{subsec:poc} apply in very broad settings. Lipschitzness is a mild assumption that is likely satisfied by many systems of interest as we will show in later sections. Therefore we see that propagation of chaos is a relatively common phenomenon for McKean-type nonlinear Markov chains, and Lipschitzness is a convenient setting for the study of many applications derived from this theory. A slower rate of $\calO(N^{-1/d})$ may well be an acceptable trade-off, at least as a starting point for further analysis.\nnewline

    Nevertheless, it is interesting to ask how one can improve these bounds and, based on the remarks above, the most obvious solution is to bypass the use of the empirical measure entirely. Transportation inequalities and their favourable tensorization properties provide one such option, and we discuss this in some detail below.

    \paragraph{Transportation Inequalities} First, let us investigate the case where $\eta_n$ satisfies a $T_2(\lambda_n)$ transportation inequality for each $n\geq 0$, i.e. that $\W_2(\nu, \eta_n)\leq \sqrt{2\lambda_n \KL(\nu\|\eta_n)}~\forall n\geq 0$ and every $\nu\in \calP_2(\R^d)$ with $\nu\ll \eta_n$. This can be used to obtain dimension-free quantitative propagation of chaos as follows: proceeding formally and neglecting issues of absolute continuity, if $\eta_n$ satisfies a $T_2(\lambda_n)$ inequality then its $q$-fold tensor product also satisfies a $T_1(\lambda_n)$ inequality on the product space $(\R^d)^{\otimes q}$, i.e. $\W_1(\nu, \eta_n^{\otimes q})\leq \sqrt{2\lambda_n \KL(\nu\|\eta_n^{\otimes q})}$ \cite{gozlan2009characterization}. But then we can apply an inequality of Csizar for exchangeable measures \cite[Lemma 8.5.1]{del2004feynman} to obtain
    \[
        \W_1(\eta^{N,q}_n, \eta_n^{\otimes q})\leq \sqrt{2\lambda_n \KL(\eta^{N,q}_n\|\eta_n^{\otimes q})} \lesssim \frac{\sqrt{q}}{\sqrt{N}}\sqrt{\lambda_n \KL(\eta^{N,N}_n\|\eta_n^{\otimes N})}.
    \]This is a substantial improvement over Theorem~\ref{thm:poc_general} because it does not explicitly depend on $d$ (although $\lambda_n\KL(\eta^{N,N}_n, \|\eta_n^{\otimes N})$ may depend on $d$ implicitly), and $\calO(\sqrt{q/N})$ is constant in $q/N$, which is closer to the optimal rates from \cite{lacker2023hierarchies}. However, this approach has the significant drawback that proving $T_2(\lambda)$ inequalities is generally hard, and it seems unlikely to be tractable for many systems of interest. One sufficient condition is if $\eta_n$ satisfies a logarithmic Sobolev inequality, although this would be nontrivial to show for each $\eta_n$ without special knowledge of the system. \nnewline

    On the other hand, $T_1(\lambda)$ inequalities have a much more tractable sufficient condition in the form of existence of square-exponential moments \cite{bolley2005weighted}, i.e., when $\E[\exp(\alpha \|X\|^2)]<\infty$ for $\alpha>0$. Moreover, the moment control techniques developed in this work can be adapted to proving the existence of square-exponential moments for the time-marginals $\eta_n$ -- indeed, this seems likely to be true for both applications we consider in Sections~\ref{sec:mv}~\&~\ref{sec:fkm}. The trade-off for more tractable sufficient conditions is that, under the $T_1(\lambda)$ condition, we now have \emph{linear} tensorization in $q$, not constant:
    \[
        \W_1(\eta_n^{N,q}, \eta^{\otimes q}_n)\leq \sqrt{2\lambda_n q\KL(\eta^{N,q}_n, \eta^{\otimes q}_n)}\lesssim \frac{q}{\sqrt{N}}\sqrt{\lambda_n \KL(\eta^{N,N}_n\|\eta_n^{\otimes N})}.
    \]This is still a net improvement over our current results as the scaling in $N$ is no longer dependent on $d$, and the worst-case scaling in $q$ is $\calO(q)$ versus $\calO(q^3)$ previously.\nnewline

    In both cases, the issue of $\KL(\eta^{N,N}_n\|\eta^{\otimes N}_n)$ remains. However this may not to be too much of an issue for many interesting applications. Showing
    \[
        \diff{\eta^{N,N}_n}{\eta^{\otimes N}_n}(x)\propto\exp  H_N(x)
    \]with $N\mapsto H_N(x)$ is $\calO(1)$ will suffice, where $H_N$ is an ``interaction Hamiltonian''. For example, considering the system in Section~\ref{sec:fkm}, in \cite[Theorem~8.2.3]{del2004feynman} it is shown that $\eta^{N,N}_n\ll \eta_n^{\otimes N}$ and furthermore they derive  $\calO(1)$ bounds w.r.t $N$ in-terms of regularity of $K_\bullet$, which become uniform-in-time as soon as $K_\bullet$ satisfies sufficient regularity conditions \cite[Corollary~8.5.1]{del2004feynman}. This essentially shows that our $\calO(q/N^{1/d}\vee q^3/N))$ results in Section~\ref{sec:fkm} can be be strengthened to $\calO(q/\sqrt{N})$, at the cost of using deep knowledge of the system of interest versus general tools such as Lipschitz regularity.\nnewline

    \subsection{Moments}
    We finish this section with a discussion on establishing moment bounds of the form $\E[\|X^1_n\|^p]\leq M^{(p)}_n$ as required by Theorem~\ref{thm:poc_general}. Many popular examples of nonlinear Markov kernels $K_\eta$ have a structure which consists of a ``linear'' part and a nonlocal ``interaction part'' which depends ultimately on a integral w.r.t. $\eta$. Notably, \emph{both} applications from Sections~\ref{sec:mv}~\&~\ref{sec:fkm}, which have very different origins, satisfy this structure. Below, we show how popular ``drift criterion'' or ``Lyapunov function'' techniques from linear Markov theory can be used to establish moment control when $K_\eta$ satisfies this form.

    \begin{assumption}[Lyapunov Moment Condition]\label{assump:lyap_mom}
        Let $V:\R^d\to[0,\infty)$ be a measurable function. Assume that there are $a_n,b_n,c\in [0,\infty)$ with $a_n+b_n,c>0$, such that
        \[
            K^{(n)}_\eta V(x) \leq a_n V(x) + b_n\eta(V) + c.
        \]
    \end{assumption}

    \begin{proposition}\label{prop:lyap_mom_mf}
        Assume Assumption~\ref{assump:lyap_mom} holds and that $\eta_0(V)<\infty$. Then for $n\geq 0$, $\eta_{n}$ from \eqref{eq:mean_field} satisfies
        \[
            \eta_{n}(V)\leq A(n) \eta_0(V) + c\sum_{k=1}^{n} \prod_{j=k+1}^n (a_j + b_j)
        \]where
        \[
            A(k):=\prod_{i=1}^k (a_i + b_i),~~~k=1,2,\dots.
        \]
    \end{proposition}
    \begin{proof}
        For $n=1$, we have
        \[
            \eta_1(V)\leq (a_1 + b_1)\eta_0(V) + c  = (a_1 + b_1)\eta_0 + c\sum_{k=1}^1\prod_{j=2}^1 (a_j + b_j).
        \]Hence assuming true for $n$, for $n+1$ we have
        \begin{align*}
            \eta_{n+1}(V)&=\eta_n(K^{(n)}_{\eta_n}V)\\
            &\leq (a_{n+1} + b_{n+1})\eta_n(V) + c \\
            &= (a_{n+1} + b_{n+1})\left[A(n)\eta_0(V) + c\sum_{k=1}^n \prod_{j=k+1}^n (a_j + b_j)\right] + c\\
            &= A(n+1)\eta_0(V) + c\sum_{k=1}^n \prod_{j=k+1}^{n+1}(a_j + b_j) + c\\
            &=A(n+1)\eta_0(V) + c\sum_{k=1}^{n+1} \prod_{j=k+1}^{n+1}(a_j + b_j)
        \end{align*}
        since the $k=n+1$-th term is 1:
        \[
            \prod_{j=(n+1)+1}^{n+1}(a_j + b_j) := 1.
        \]
    \end{proof}

    Now we extend this approach to the moments of the particle system \eqref{eq:ips}, which is actually what is required for Theorem~\ref{thm:poc_general}. We leverage the useful fact that, while in general $\eta^N_n\neq \Phi_n[\eta^N_{n-1}]$, it is true in the \emph{weak sense}, i.e., $\E[\eta^N_n(f)]=\E[\Phi_n[\eta^N_{n-1}](f)]$ for the same reason as Lemma~\ref{lem:emp_mom}. Note that this is only a one-step equality, i.e., we do \emph{not} have $E[\eta^N_n(f)]=\E[\Phi_n[\Phi_{n-1}[\eta^N_{n-2}]](f)]$. However, we can use the fact that the inequality in Assumption~\ref{assump:lyap_mom} is \emph{recursive} to say something about the particle dynamics, in expectation. Therefore, a one-step equality $\E[\eta^N_n(f)]=\E[\Phi_n[\eta^N_{n-1}](f)]$ plus a recursive inequality on $K^{(n)}_\eta$ allows us to treat the moments of the particle system like those of the (simpler) mean-field system in Proposition~\ref{prop:lyap_mom_mf}.

    \begin{proposition}\label{prop:lyap_mom_ips}
        Suppose that Assumption~\ref{assump:lyap_mom} holds, and that $\E[\eta^N_0(V)]<\infty$. Then
        \[
            \E[\Phi_{n+1}[\eta^N_n](V)]\leq A(n+1)\E[\eta^N_0(V)] + c\sum_{k=1}^{n+1} \prod_{j=k+1}^{n+1} (a_j + b_j).
        \]
    \end{proposition}
    \begin{proof}
        Firstly, for $\calF^N_n:=\sigma(X^1_n, \dots, X^N_n)$, we have almost surely
        \begin{align*}
            \E[\Phi_{n+1}[\eta^N_n](V)|\calF^N_n] &= \eta^N_n\left(K^{(n+1)}_{\eta^N_n}V\right) \\
            &\leq a_{n+1}\eta^N_n(V) + b_{n+1}\eta^N_n(V) + c \\
            &= (a_{n+1} + b_{n+1})\eta^N_n(V) + c
        \end{align*}
        so that by the Law of Total Expectation,
        \[
            \E[\Phi_{n+1}[\eta^N_n](V)]\leq (a_{n+1} + b_{n+1})\E[\eta^N_n(V)] + c.
        \]Using exchangeability, the Law of Total Expectation again, and the fact that $X^i_n \sim K^{(n)}_{\eta^N_{n-1}}(X^i_{n-1}, \bullet)$, we have
        \begin{align*}
            \E[\eta^N_n(V)] &= \E\left[\frac{1}{N}\sum_{i=1}^N V(X^i_n)\right]=\E[V(X^1_n)]=\E[K^{(n)}_{\eta^N_{n-1}}V(X^1_{n-1})]\\
            &\leq a_n \E[V(X^1_{n-1})] + b_n \E[\eta^N_{n-1}(V)] + c\\
            &= (a_n + b_n) \E[\eta^N_{n-1}(V)] + c.
        \end{align*}
        Hence, proceeding as in Proposition~\ref{prop:lyap_mom_mf}, we get
        \[
            \E[\eta^N_n(V)]\leq A(n) \E[\eta^N_0(V)] + c\sum_{k=1}^{n} \prod_{j=k+1}^n (a_j + b_j)
        \]so that
        \begin{align*}
            \E[\Phi_{n+1}[\eta^N_n](V)]&\leq (a_{n+1} + b_{n+1})\E[\eta^N_{n}(V)] + c\\
            &\leq A(n+1)\E[\eta^N_0(V)] + c\sum_{k=1}^{n+1} \prod_{j=k+1}^{n+1} (a_j + b_j).
        \end{align*}
    \end{proof}

    The following is an immediate corollary of the above results.
    \begin{corollary}
        Suppose that Assumption~\ref{assump:lyap_mom} is satisfied, and that $\E[\eta^N_0(V)]<\infty$. Let $p\geq 1$; if there exists constants $\theta>0$ and $R\geq 0$ such that 
        \[
            \|x\|^p\leq \theta V(x)~\forall \|x\|>R
        \]then there are constants $M_n^{(p)}\in(0,\infty)$ such that
        \[
            \E[\|X^1_n\|^p]\leq M^{(p)}_n~~~\forall n=0,1,2,\dots
        \]Furthermore, if $a_n + b_n < 1~\forall n\geq 1$, then the above bound can be replaced by a uniform constant $M^{(p)}_\star$.
    \end{corollary}

    \begin{remark}
        Note that the same corollary applies instead with $ V(x)=\exp(\alpha \|x\|^2);~\alpha>0$. In this case, we obtain control over square-exponential moments, which in turn implies finiteness of \emph{all} moments, and also a $T_1$ transportation inequality \cite{bolley2005weighted}, as discussed in Section~\ref{subsec:rates}.
    \end{remark}

    \section{Uniform Propagation of Chaos}\label{sec:uniform_propagation_of_chaos}

    Below is an immediate consequences of Theorem~\ref{thm:poc_general} when $K^{(n)}_\bullet$ is a contraction for every $n \geq 1$.
    \begin{corollary}[Uniform Propagation of Chaos]\label{coro:unif_poc}
        Suppose $\tau_1(K^{(n)}_\bullet)\leq \tau_1^\star<1~\forall n\geq 1$ and $M^{(p)}_k\leq M^{(p)}_\star<\infty~\forall k\geq 0$. Then 
        \[
            \W_p(\eta^{N,q}_n, \eta^{\otimes q}_n)\leq C_{p, d}\cdot \frac{\left(M^{(p)}_\star\right)^{1/p}}{1-\tau_1^\star}\frac{q}{N^{1/d}}.
        \]
    \end{corollary}

    \subsection{Modified Wasserstein Distance}
    
    \begin{definition}[Modified 1-Wasserstein Distance]
        Suppose there is a concave, increasing function $\rho:[0,\infty)\to[0,\infty)$ satisfying $\rho(0)=0$ and
        \[
            \ul{c}r\leq \rho(r) \leq \ovl{C}r
        \]for constants $\ul{c},\ovl{C}\in (0,\infty)$. Then
        \[
            d_\rho(x,y):=\rho(\|x-y\|)
        \]defines a metric on $\R^d$ and we may define the associated 1-Wasserstein distance on $\calP_\rho(\R^d)$
        \[
            \W_{\rho}(\mu,\nu):=\inf_{\gamma\in \calC(\mu,\nu)}\int d_\rho(x,y)\gamma(\dd x, \dd y)=\inf_{\gamma\in \calC(\mu,\nu)}\int \rho(\|x-y\|)\gamma(\dd x, \dd y)
        \]where we have defined
        \[
            \M_\rho(\mu):=\int \rho(\|x\|)\mu(\dd x)
        \]and 
        \[
            \calP_\rho(\R^d):=\{\mu\in\calP(\R^d) \st \M_\rho(\mu)<\infty\}.
        \]
    \end{definition}
    It is easy to see that $\W_\rho$ and $\W_1$ are equivalent and that $\M_1(\mu)<\infty$ iff $\M_\rho(\mu)<\infty$. However $\W_\rho$ offers more flexibility to tailor the metric $d_\rho(x,y)$ to the kernel of interest and therefore obtain a contraction. This would lead straightforwardly to a uniform propagation of chaos in the case of a time-homogeneous kernel or a uniform-in-$n$ bound on the contraction coefficients. This approach is now a standard technique in the study of Markov chains \cite{hairer2011yet} and nonlinear diffusions \cite{eberle2019quantitative,durmus2020elementary}.
    \begin{corollary}
        Suppose that the assumptions of Theorem~\ref{thm:poc_general} are satisfied, and that $K^{(n)}_\bullet$ is a contraction in $\W_\rho$, i.e.
        \[
            \W_\rho(\mu K^{(n)}_\mu,\nu K^{(n)}_\nu)\leq \tau_\rho(K^{(n)}_\bullet)\W_\rho(\mu,\nu)
        \]where $\tau_\rho(K^{(n)}_\bullet) \leq \tau_\rho^\star<1$ and $M^{(p)}_n\leq M^{(p)}_\star<\infty$ for all $n\geq 1$. Then
        \[
            \E[\W_1(\eta^{N}_n, \eta_n)]\leq C_{p, d}\frac{\ovl{C}}{\ul{c}}\cdot \frac{1}{N^{1/d}}\frac{(M^{(p)}_\star)^{1/p}}{1-\tau_\rho^\star}.
        \]
    \end{corollary}
    \begin{proof}
        Firstly, we note that $\calP_1(\R^d)=\calP_\rho(\R^d)$ so that
        \begin{align*}
            \tau_1(K^{(n)}_\bullet)&=\sup_{\mu,\nu\in\calP_1(\R^d),~\mu\neq\nu}\frac{\W_1(\mu K^{(n)}_\mu, \nu K^{(n)}_\nu)}{\W_1(\mu,\nu)}\\
            &=\sup_{\mu,\nu\in\calP_\rho(\R^d),~\mu\neq\nu}\frac{\W_1(\mu K^{(n)}_\mu, \nu K^{(n)}_\nu)}{\W_1(\mu,\nu)}\\
            &\leq \sup_{\mu,\nu\in\calP_\rho(\R^d),~\mu\neq\nu}\frac{\ovl{C}}{\ul{c}}\cdot \frac{\W_\rho(\mu K^{(n)}_\mu, \nu K^{(n)}_\nu)}{\W_\rho(\mu,\nu)}\\
            &= \frac{\ovl{C}}{\ul{c}}\tau_\rho(K^{(n)}_\bullet).
        \end{align*}
        Hence, we can use this to improve the estimate from Theorem~\ref{thm:poc_general}:
        \begin{align*}
            \E[\W_1(\eta^{N}_n, \eta_n)]&\leq C_{r, d}\cdot \frac{1}{N^{1/d}}\sum_{k=0}^n \left(M^{(p)}_k\right)^{1/p}\tau_1(\Phi_{n,k})\\
            &\leq C_{p, d}\frac{\ovl{C}}{\ul{c}}\cdot \frac{1}{N^{1/d}}\sum_{k=0}^n \left(M^{(p)}_k\right)^{1/p}\tau_\rho(\Phi_{n,k})\\
            &\leq C_{p, d}\frac{\ovl{C}}{\ul{c}}\cdot \frac{1}{N^{1/d}}\left(M^{(p)}_\star\right)^{1/p}\sum_{k=0}^n (\tau_\rho^{\star})^{n-k}\\
            &\leq C_{p, d}\frac{\ovl{C}}{\ul{c}}\cdot \frac{1}{N^{1/d}}\frac{\left(M^{(p)}_\star\right)^{1/p}}{1-\tau_\rho^\star}.
        \end{align*}
        The other results from Section~\ref{sec:propagation_of_chaos} apply \emph{mutatis mutandis}.
    \end{proof}

    \subsection{Geometric Ergodicity}
    Finally, we explore the possibility of converting a non-uniform propagation of chaos result, e.g. from Theorem~\ref{thm:poc_good}, into a uniform one. Specifically, we show that \emph{geometric ergodicity} of the mean field and interacting particle systems, plus non-uniform propagation of chaos, can result in uniform propagation of chaos. This approach was explored in \cite{guillin2021uniform} in continuous time for the 2-Wasserstein distance.
    \begin{assumption}\label{assump:ergodic}
        Let $K_\bullet$ be time-homogeneous.
        \begin{assumptions}[E]
            \item Assume that $K_\bullet$ satisfies Assumption~\ref{assump:R1}, that $M_k^{(r)}\leq M_\star^{(p)}<\infty$, and that \text{$\tau:=\tau_1(K_\bullet)>1$} so that by Proposition~\ref{prop:emp_main}
            \[
                \E[\W_1(\eta^N_n, \eta_n)]\leq \wtilde{C}_{p,d} N^{-1/d}\tau^n
            \]for some $\wtilde{C}_{p,d}>0$ independent of $N$.
            \item There is $\alpha<1,~C<\infty$, an invariant measure $\eta_\infty\in\calP_1(\R^d)$ for \eqref{eq:mean_field} and for every $N\geq 1$, there is an invariant measure $\eta^{N,N}_\infty\in\calP_1((\R^d)^{\otimes N})$ for \eqref{eq:ips} such that we have
            \begin{align*}
                \W_1(\eta_n,\eta_\infty)&\leq C\alpha^n \W_1(\eta_0,\eta_\infty)\\
                \E[\W_1(\eta^N_n,\eta^{N}_\infty)]&\leq C\alpha^n \E[\W_1(\eta^{N}_0,\eta^{N}_\infty)]
            \end{align*}
            where $\eta^N_\infty:=m(\ovl{X}^N_\infty)$ with $\ovl{X}^N_\infty\sim \eta^{N,N}_\infty$.
            \item $\eta_\infty$ has finite $p$-th moment for some $p>1$.
        \end{assumptions}
    \end{assumption}

    \begin{proposition}\label{prop:invar}
        Suppose Assumption~\ref{assump:ergodic} holds. Then there exists a constant $C'>0$ independent of $N$ such that
        \[
            \E[\W_1(\eta^N_\infty, \eta_\infty)]\leq C'N^{-1/d}.
        \]
    \end{proposition}
    \begin{proof}
        For any trajectories $\eta_n,\eta^{N,N}_n$ following dynamics \eqref{eq:mean_field}, \eqref{eq:ips} respectively, with initial conditions $\eta_0=\eta_\infty$, $X^i_0\iid \eta_0$, we have
        \begin{align*}
            \E[\W_1(\eta^N_\infty,\eta_\infty)]&\leq \E[\W_1(\eta^N_\infty,\eta^N_n)] + \E[\W_1(\eta^N_n,\eta_n)] + \W_1(\eta_n,\eta_\infty)\\
            &\leq C\alpha^n\E[\W_1(\eta^N_\infty,\eta^N_0)] + \wtilde{C}_{p,d}N^{-1/d}\tau^n + \underbrace{C\alpha^n \W_1(\eta_0,\eta_\infty)}_{=0}\\
            &=C\alpha^n\E[\W_1(\eta^N_\infty,\eta^N_0)] + \wtilde{C}_{p,d}N^{-1/d}\tau^n\\
            &\leq C\alpha^n\E[\W_1(\eta^N_\infty,\eta_0)] + C\alpha^n\E[\W_1(\eta^N_0,\eta_0)] + \wtilde{C}_{p,d}N^{-1/d}\tau^n\\
            &=C\alpha^n\E[\W_1(\eta^N_\infty,\eta_\infty)] + C\alpha^n\E[\W_1(\eta^N_0,\eta_\infty)] + \wtilde{C}_{p,d}N^{-1/d}\tau^n.
        \end{align*}
        Note that $\eta^N_0\neq \eta^N_\infty$, but rather $\eta^N_0=\wtilde{\eta}^N_\infty=:m(\wtilde{X}^1_\infty, \dots, \wtilde{X}^N_\infty)$ where each $\wtilde{X}^i_\infty\iid \eta_\infty$. However, we do have from \cite{fournier2015rate} that 
        \[
            \E[\W_1(\wtilde{\eta}^N_\infty,\eta_\infty)]\leq C_{p,d}\M_p(\eta_\infty)N^{-1/d}
        \]so that
        \[
            \E[\W_1(\eta^N_\infty,\eta_\infty)]\leq C\alpha^n\E[\W_1(\eta^N_\infty,\eta_\infty)] + \wtilde{C}_{p,d}'N^{-1/d}(\tau^n + 1).
        \]Then picking $n=\lceil \log(2C)/\log(1/\alpha)\rceil$ we get
        \[
            \E[\W_1(\eta^N_\infty,\eta_\infty)]\leq 2\wtilde{C}N^{-1/d}(2C)^{\log(\tau)/\log(1/\alpha)}.
        \]
    \end{proof}

    \begin{theorem}
        Suppose Assumption~\ref{assump:ergodic} holds, and that $\eta^{N,N}_0=\eta_0^{\otimes N}$. Then there is $\kappa\in (0, 1/d)$ and $C'$ independent of $n,N$ s.t. $\forall n\geq 0$ and $N\geq 1$ we have
        \[
            \E[\W_1(\eta^N_n, \eta_n)]\leq \frac{C'}{N^\kappa}.
        \]
    \end{theorem}
    \begin{proof}
        Using the triangle inequality, plus Proposition~\ref{prop:invar}, we get
        \begin{align*}
            \E[\W_1(\eta^{N}_n, \eta_n)]&\leq \E[\W_1(\eta^{N}_n, \eta^{N}_\infty)] + \E[\W_1(\eta^{N}_\infty, \eta_\infty)] + \W_1(\eta_\infty, \eta_n)\\
            &\leq C\alpha^n\E[\W_1(\eta^{N}_0,\eta^{N}_\infty)] + \wtilde{C} N^{-1/d} + C\alpha^n\W_1(\eta_0, \eta_\infty).
        \end{align*}
        Using the fact that $X^i_0\iid \eta_0$,
        \begin{align*}
            \E[\W_1(\eta^{N}_0,\eta^{N}_\infty)]&\leq \E[\W_1(\eta^{N}_0,\eta_0)] + \W_1(\eta_0,\eta_\infty) + \E[\W_1(\eta_\infty,\eta^{N}_\infty)]
        \end{align*}
        and using \cite[Theorem~1]{fournier2015rate} again on the first term and Proposition~\ref{prop:invar} on the last term we get
        \begin{align*}
            \E[\W_1(\eta^{N}_0,\eta^{N}_\infty)]&\leq C''N^{-1/d} + \W_1(\eta_0,\eta_\infty) + \wtilde{C}N^{-1/d}\\
            &= C'''N^{-1/d} + \W_1(\eta_0,\eta_\infty).
        \end{align*}
        Therefore inserting this above, we are left with
        \[
            \E[\W_1(\eta^{N}_n, \eta_n)]\leq C_0\alpha^n + C_1N^{-1/d}
        \]for some constants $C_0,C_1$ independent of $n,N$.

        To conclude, we balance this estimate, which is good for large $n$ but has a term independent of $N$, against the estimate from Theorem~\ref{thm:poc_good}, which behaves well in $N$ for all fixed $n$ but explodes as $n\to\infty$. In precise terms, we have
        \[
            \E[\W_1(\eta^{N}_n, \eta_n)]\leq C_2\left( (\alpha^n + N^{-1/d})\wedge \tau^n N^{-1/d}\right)
        \]hence for $n_0:=(1/d)\log(N)/\log(\tau/\alpha)$, distinguishing between $n<n_0$ and $n>n_0$ establishes the proof, with
        \[
            \kappa:=\frac{1}{d}\frac{\log(1/\alpha)}{\log(\tau/\alpha)}.
        \]
    \end{proof}

    Note that $\E[\W_1(\eta^N_n, \eta_n)]$ being geometrically ergodic, independent of $N$ no less, is a strong assumption. It is stronger than asking the same of $\eta^{N,N}_n$. It would be much more desirable to study $\eta^{N,N}_n$ as this is a \emph{bona fide} linear Markov chain, and hence its ergodicity is very well understood (although dependence on $N$ is another matter). However, all we know at this level of generality is that
    \[
        \E[\W_1(\eta^N_n,\eta^N_{\infty})]\leq \frac{1}{\sqrt{N}}\W_1(\eta^{N,N}_n, \eta^{N,N}_\infty)
    \]using the $1/\sqrt{N}$-Lipschitzness of the map $x:=(x^1,\dots,x^N)\mapsto m(x)$. Assuming instead that $\eta^{N,N}_n$ is uniformly-in-$N$ ergodic, we would then get
    \[
        \E[\W_1(\eta^N_n,\eta^N_\infty)]\leq \frac{1}{\sqrt{N}}C\alpha^n \W_1(\eta^{N,N}_0, \eta^{N,N}_\infty).
    \]However, this is incompatible with the strategy of Proposition~\ref{prop:invar} as we need to relate the right side back to the left side in order to obtain a self-contained bound.

    \begin{corollary}
        Under Assumption~\ref{assump:ergodic}, we have the uniform propagation of chaos
        \[
            \W_1(\eta^{N,q}_n, \eta_n^{\otimes q})\leq C'\frac{q}{N^{\kappa}} + 2\frac{q^3}{N}.
        \]
    \end{corollary}
    This is $\calO(q/N^{\kappa} \vee q^3/N)
    )$ which is worse than the coupling bound in Theorem~\ref{thm:poc_good}. However, note that above we explicitly assumed that the Lipschitz constant was greater than one so we cannot plug this back into Theorem~\ref{thm:poc_good} and we must use the more general, but worse, approach in Theorem~\ref{thm:poc_general} that doesn't use additional regularity. One exception is that if Assumption~\ref{assump:R2} is satisfied with $c<1$ but $c+C>1$, this would result in $\calO(q/N^\kappa)$ bounds.

	\section{Euler-Maruyama Scheme for McKean-Vlasov SDE}\label{sec:mv}
	In this section, we use the tools developed thus far to study propagation of chaos for a family of nonlinear Markov kernels associated to Euler-Maruyama discretizations of the following McKean-Vlasov nonlinear stochastic differential equation, recalled from \eqref{eq:mv_sde}
	\begin{equation}\label{eq:mv_sde_again}
		\dd X_t = - \nabla V(X_t)\dd t - \nabla W* \eta_t(X_t) \dd t + \sqrt{2}\dd B_t.
	\end{equation}
	This equation, and continuous-time propagation of chaos techniques for its associated particle system, are discussed in Section~\ref{subsec:related_work}. In particular, the numerical analysis of these nonlinear diffusions was originally studied in \cite{bossy1997stochastic}, and we use the same setup, although our analysis is different. Hence consider the simple first-order, explicit Euler-Maruyama scheme with discretization size $\delta>0$ associated to \eqref{eq:mv_sde_again}
	\begin{equation}\label{eq:mv_mc}
		X_{n + 1} = X_n - \delta \nabla V(X_n) - \delta\nabla W * \eta_n(X_n) + \sqrt{2 \delta}Z_n;~~~\eta_n:=\law(X_n),~Z_n\iid \calN(0, I_d)
	\end{equation}
	which defines a nonlinear Markov chain having the nonlinear Markov kernel
	\begin{equation}\label{eq:mv_kernel}
		K^\delta_\eta(x, \dd y) := \calN(x - \delta\nabla V(x) - \delta\nabla W * \eta(x), 2\delta I_d)(\dd y)
	\end{equation}
	where $\calN(m, \Sigma)$ is the $d$-dimensional multivariate normal with mean vector $m$ and covariance matrix $\Sigma$. Below, we show how the techniques developed in Section~\ref{sec:propagation_of_chaos} can be readily applied to prove propagation of chaos in the case that $V + W$ is uniformly convex. Note that $\delta$ will be fixed for this analysis, so we will suppress it in the notation when there is no chance of confusion.

	\begin{assumption}\label{assump:mv}
		Let $V,W:\R^d\to\R$. 
		\begin{assumptions}[MV]
			\item \label{assump:mv1} Assume that $\hess(V)\geq \lambda_VI_d$ and $\hess(W)\geq \lambda_W I_d$  for $\lambda_V, \lambda_W\in\R$.
			\item \label{assump:mv2} Assume that $\nabla V$ is $C_V$-Lipschitz and that $\nabla W$ is $C_W$-Lipschitz.
			\item \label{assump:mv3} $W$ is an even function.
		\end{assumptions}
	\end{assumption}

	Now we prove a basic lemma on the properties of the interaction term.
	\begin{lemma}\label{lem:conv_lip}
		Let $\mu,\nu\in\calP_1(\R^d)$.
		\begin{enumerate}
			\item Suppose Assumption~\ref{assump:mv1}, i.e.  $W:\R^d\to \R$ is s.t. $\hess(W)\geq \lambda_WI_d$ for some $\lambda_W\in\R$. Then for $\mu\in \calP(\R^d)$
			\[
				\ip{x - y}{\nabla W * \mu(x) - \nabla W*\mu(y)}\geq \lambda_W \|x - y\|^2
			\]
			\item Suppose Assumption~\ref{assump:mv2}, i.e. $\nabla W$ is Lipschitz-continuous with constant $C_W$. Then
			\[
				\|\nabla W * \mu(x) - \nabla W*\nu(y)\|\leq C_W\|x - y\| + C_W\W_1(\mu,\nu).
			\]
			\item Suppose Assumptions~\ref{assump:mv2},\ref{assump:mv3}, i.e. $\nabla W$ is Lipschitz-continuous with constant $C_W$, and that $\M_2(\eta)<\infty$. Then
			\[
				\|\nabla W * \eta(x)\|^2 \leq 2C_W^2\|x\|^2 + 2C_W^2 \M_2(\eta).
			\]
		\end{enumerate}
	\end{lemma}
	\begin{proof}
		\hfill 
		\begin{enumerate}
			\item We have
			\begin{align*}
				&\ip{x - y}{\nabla W*\mu(x) - \nabla W*\mu(y)}\\
				&= \E_\mu[\ip{x - y}{\nabla W(x - Z) - \nabla W(y - Z)}]\\
				&= \E_\mu[\ip{(x - Z) -(y - Z)}{\nabla W(x - Z) - \nabla W(y - Z)}]\\
				&\geq \lambda_W\E_\mu[ \|(x - Z) - (y - Z)\|^2] = \lambda_W \|x - y\|^2
			\end{align*}
			\item Consider
			\begin{align*}
				&\|\nabla W*\mu(x) - \nabla W*\nu(y)\|\\
				&\leq \|\nabla W*\mu(x) - \nabla W*\mu(y)\| + \|\nabla W*\mu(y) - \nabla W*\nu(y)\|\\
				&= \left\|\int \nabla W(x - z)\mu(\dd z) - \int \nabla W(y - z')\mu(\dd z')\right\|\cdots\\
				&+ \left\|\int \nabla W(y - z)\mu(\dd z) - \int \nabla W(y - z')\nu(\dd z')\right\|\\
				&= \left\|\int [\nabla W(x - z) - \nabla W(y-z)]\mu(\dd z)\right\| + \left\|\int \nabla W(y - z)\mu(\dd z) - \int \nabla W(y - z')\nu(\dd z')\right\|
			\end{align*}
			The first term can be bounded using Lipschitzness of $W$ via
			\[
				\left\|\int [\nabla W(x - z) - \nabla W(y-z)]\mu(\dd z)\right\|\leq \int\|\nabla W(x - z) - \nabla W(y - z)\|\mu(\dd z)\leq C_W\|x - y\|.
			\]For the second term, let $\gamma\in\calC(\mu,\nu)$ be any coupling of $\mu,\nu$; then
			\begin{align*}
				&\left\|\int \nabla W(y - z)\mu(\dd z) - \int \nabla W(y - z')\nu(\dd z')\right\|\\
				&= \left\|\int \nabla W(y - z)\gamma(\dd z, \dd z') - \int \nabla W(y - z')\gamma(\dd z, \dd z')\right\|\\
				&=\left\|\int [\nabla W(y - z) - \nabla W(y - z')]\gamma(\dd z, \dd z')\right\|\\
				&\leq \int \|\nabla W(y - z) - \nabla W(y - z')\|\gamma(\dd z, \dd z')\\
				&\leq C_W \int \|z - z'\|\gamma(\dd z, \dd z').
			\end{align*}
			Taking the infimum over all such couplings yields
			\[
				\|\nabla W*\mu(y) - \nabla W*\nu(y)\|\leq C_W \W_1(\mu,\nu)
			\]
			\item Using Lipzchitzness of $\nabla W$ and evenness of $W$, which implies $\nabla W(0)=0$, we get
			\begin{align*}
				\|\nabla W * \eta(x)\|^2&= \left\|\int \nabla W(x - y)\eta(\dd y)\right\|^2\\
				&\leq \int \|\nabla W(x - y)\|^2\eta(\dd y) \\
				&\leq C_W^2 \int \|x-y\|^2 \eta(\dd y)\\
				&\leq C_W^2 2\|x\|^2 + 2C_W^2 \M_2(\eta).
			\end{align*}
		\end{enumerate}
	\end{proof}

	\subsection{Regularity}
	Now we can prove a central result establishing the regularity of $(x,\mu)\mapsto K_\mu(x,\bullet)$ in $\W_2$, from which estimating the contraction coefficient of $K_\bullet$ will follow easily.

	\begin{proposition}\label{prop:mv_wass}
		Suppose that Assumption~\ref{assump:mv} holds, and let $\mu,\nu\in \calP_1(\R^d)$. Then
		\[
			\W_2(K_\mu(x, \bullet), K_\nu(y, \bullet))\leq \sqrt{1 - 2\delta(\lambda_V + \lambda_W) + \delta^2(C_V + C_W)^2}\|x - y\| + \delta C_W\W_1(\mu,\nu).
		\]
	\end{proposition}
	\begin{proof}
		Since $K_\eta(x, \bullet)$ is Gaussian, we can use the formula for the optimal $\W_2$-distance between two multivariate normal distributions to conclude that
		\[
			\W_2(K_\mu(x, \bullet), K_\nu(y, \bullet))=\|x -\delta \nabla V(x) - \delta\nabla W*\mu(x) - (y - \delta\nabla V(y) - \delta\nabla W*\nu(y))\|.
		\]We will use the triangle inequality to separately estimate
		\[
			\W_2(K_\mu(x, \bullet), K_\nu(y, \bullet))\leq \underbrace{\W_2(K_\mu(x, \bullet), K_\mu(y, \bullet))}_{(1)} + \underbrace{\W_2(K_\mu(y, \bullet), K_\nu(y, \bullet))}_{(2)}.
		\]For $(1)$, we have
		\begin{align*}
			\W^2_2(K_\mu(x, \bullet), K_\mu(y, \bullet))&= \|x - y - \delta(\nabla V(x) + \nabla W*\mu(x) - (\nabla V(y) + \nabla W*\mu(y))\|^2\\
			&= \|x - y\|^2 - 2\delta\ip{x - y}{\nabla V(x) + \nabla W*\mu(x) - (\nabla V(y) + \nabla W*\mu(y))} +\cdots \\
			& \cdots +\delta^2 \|\nabla V(x) + \nabla W*\mu(x) - (\nabla V(y) + \nabla W*\mu(y))\|^2.
		\end{align*}
		The last term can be straightforwardly estimated by $\delta^2(C_V + C_W)^2\|x-y\|^2$ from the Lipschitz assumptions and Lemma~\ref{lem:conv_lip}. For the middle term, we have
		\begin{align*}
			&2\delta\ip{x - y}{\nabla V(x) + \nabla W*\mu(x) - (\nabla V(y) + \nabla W*\mu(y))}\\
			&= 2\delta \ip{x - y}{\nabla V(x) - \nabla V(y)} + 2\delta \ip{x - y}{\nabla W*\mu(x)  - \nabla W*\mu(y))}\\
			&\geq 2\delta(\lambda_V + \lambda_W)\|x - y\|^2
		\end{align*}
		using Lemma~\ref{lem:conv_lip}. Hence for $(1)$ we have
		\[
			\W_2(K_\mu(x, \bullet), K_\mu(y, \bullet)\leq \sqrt{1 - 2\delta(\lambda_V + \lambda_W) + \delta^2(C_V + C_W)^2}\|x - y\|.
		\]For $(2)$, we can use Lemma~\ref{lem:conv_lip} to obtain
		\begin{align*}
			\W_2(K_\mu(y, \bullet), K_\nu(y, \bullet)) = \delta\|\nabla W*\mu(y) - \nabla W*\nu(y)\|\leq \delta C_W\W_1(\mu,\nu).
		\end{align*}
		Hence putting these together yields the result.
	\end{proof}

	Remarkably, the estimate above applies equally to $\W_1$ as $\W_2$.
	\begin{corollary}
		Under the conditions of the Proposition~\ref{prop:mv_wass}, we have
		\[
			\W_1(K_\mu(x, \bullet), K_\nu(y, \bullet))\leq \sqrt{1 - 2\delta(\lambda_W + \lambda_V)+ \delta^2(C_V + C_W)^2}\|x - y\| + \sqrt{2}\delta C_W\W_1(\mu,\nu).
		\]
	\end{corollary}
	\begin{proof}
		Just use
		\[
			\W_1(K_\mu(x, \bullet), K_\nu(y, \bullet))\leq \W_2(K_\mu(x, \bullet), K_\nu(y, \bullet)).
		\]
	\end{proof}

	Finally, Lemma~\ref{lem:tau} directly implies that we can estimate the contraction coefficient of $K^\delta_\bullet$ using the results above.
	\begin{theorem}\label{thm:mv_lip}
		Let $K^\delta_\bullet$ be the Markov kernel \eqref{eq:mv_kernel}, and suppose that Assumption~\ref{assump:mv} holds. Then 
		\begin{equation}\label{eq:tau_mv}
			\tau_1(K^\delta_\bullet) \leq \sqrt{1 - 2\delta(\lambda_V + \lambda_W) + \delta^2(C_V + C_W)^2} + \sqrt{2}\delta C_W.
		\end{equation}
	\end{theorem}

	\subsection{Moments}
	The second ingredient necessary for the application of Theorem~\ref{thm:poc_good} is moment control. We approach this using a Lyapunov condition as discussed.

	\begin{proposition}\label{prop:mv_mom}
		Let $K^\delta_\eta$ be the nonlinear Markov kernel from \eqref{eq:mv_kernel} and suppose Assumption~\ref{assump:mv} holds. Then
		\[
			K^\delta_\eta\V_2(x)\leq a\V_2(x) + b\eta(\V_2) + c
		\]where
		\begin{align*}
			a&:=1 + \delta(2 - 2(\lambda_V +\lambda_W)) + 4\delta^2(C_V^2 + C_W^2)\\
			b&:=\delta(2 + 4\delta)C_W^2\\
			c&:= (1 + 4\delta) \delta\|\nabla V(0)\|^2 + 2d\delta.
		\end{align*}
	\end{proposition}

	\begin{proof}
		Recall first that if $Y\sim K^\delta_\eta(x, \bullet)$, then
		\[
			Y = x - \delta\nabla V(x) - \delta \nabla W * \eta(x) + \sqrt{2\delta}Z;~~~Z\sim \calN(0, I_d).
		\]We will use the fact 
		\[
			K^\delta_\eta \V_2(x)=\E_{Y\sim K^\delta_\eta(x, \bullet)}[\|Y\|^2]
		\] to work directly with the 2-norm. Moreover, we will use the fact that for $u\in\R^d,v\in\R$
		\[
			\E[\|u + vZ\|^2] = \E[\|u\|^2] + 2\E[\ip{u}{vZ}] + v^2 \E[\|Z\|^2]= \|u\|^2 + v^2\E[\|Z\|^2],
		\]which follows from linearity of expectations and that $\E[Z]=0$, to obtain
		\begin{align*}
			\E[\|Y\|^2]=& \|x\|^2 - 2\delta\ip{x}{\nabla V(x)} - 2\delta \ip{x}{\nabla W*\eta(x)}+ \cdots \\
			&+ \delta^2 \|\nabla V(x)\|^2 + \delta^2 \|\nabla W*\eta(x)\|^2 + 2\delta^2\ip{\nabla V(x)}{\nabla W*\eta(x)} + 2\delta \E[\|Z\|^2].
		\end{align*}
		Now recall that $\nabla V$ is Lipschitz so $\|\nabla V(x)\|\leq C_V\|x\| + \|\nabla V(0)\|$. Also, recall that $W$ is even by assumption $\implies\nabla W(0) = 0$, so $\|\nabla W(x)\|\leq C_W\|x\|$. From Lemma~\ref{lem:conv_lip}, we have that $\|\nabla W*\eta(x)\|^2\leq 2C_W^2 \|x\|^2 + 2C_W^2 \M_2(\eta)$.
		Also, using Cauchy-Schwarz and Young's inequality for products, we have
		\[
			\ip{\nabla V(x)}{\nabla W*\eta(x)}\leq \|\nabla V(x)\|\|\nabla W*\eta (x)\|\leq \frac{1}{2}\|\nabla V(x)\|^2 + \frac{1}{2}\|\nabla W*\eta(x)\|^2
		\]
		and
		\begin{align*}
			\ip{x}{\nabla V(x)} &= \ip{x - 0}{\nabla V(x) - \nabla V(0) + \nabla V(0)}\\
			&= \ip{x - 0}{\nabla V(x)-\nabla V(0)} + \ip{x}{\nabla V(0)}\\
			&\geq \lambda_V \|x\|^2 + \ip{x}{\nabla V(0)}.
		\end{align*}
		Moreover
		\[
			|\ip{x}{\nabla V(0)}|\leq \|x\|\|\nabla V(0)\| \leq \frac{1}{2}\|x\|^2 + \frac{1}{2}\|\nabla V(0)\|^2
		\]and
		\[
			|\ip{x}{\nabla W*\eta(0)}|\leq \|x\|\|\nabla W*\eta(0)\| \leq  \frac{1}{2}\|x\|^2 + C^2_W\M_2(\eta).
		\]
		So putting this all together we get
		\begin{align*}
			\E[\|Y\|^2] \leq& \|x\|^2 - 2\delta\ip{x}{\nabla V(x)} - 2\delta \ip{x}{\nabla W*\eta(x)} + \delta^2 \|\nabla V(x)\|^2 + \delta^2 \|\nabla W*\eta (x)\|^2 + \cdots\\
			&+ 2\delta^2\ip{\nabla V(x)}{\nabla W*\eta(x)} + 2\delta \E[\ip{Z}{Z}]\\
			\leq& \|x\|^2 - 2\delta(\lambda_V + \lambda_W)\|x\|^2 - 2\delta\ip{x}{\nabla V(0)} - 2\delta \ip{x}{\nabla W*\eta(0)} + \cdots \\
			&+ 2\delta^2\|\nabla V(x)\|^2 + 2\delta^2\|\nabla W*\eta(x)\|^2 + 2\delta d\\
			\leq &\|x\|^2 - 2\delta(\lambda_V + \lambda_W)\|x\|^2 + \delta\|x\|^2 + \delta \|\nabla V(0)\|^2 + \delta \|x\|^2 + 2\delta C^2_W\M_2(\eta) + \cdots \\
			&+ 4\delta^2C_V^2\|x\|^2 + 4\delta^2\|\nabla V(0)\|^2 + 4\delta^2C_W^2\|x\|^2 + 4\delta^2 C_W^2 M_2(\eta) + 2\delta d\\
			&= \underbrace{(1 + \delta(2 - 2(\lambda_V +\lambda_W)) + \delta^2(4C_V^2 + 4C_W^2))}_{a}\|x\|^2 +\cdots \\
			&+ \underbrace{\delta(2 + 4\delta) C_W^2}_{b}\M_2(\eta) + \underbrace{\delta(1 + 4\delta)\|\nabla V(0)\|^2 + 2d\delta}_{c}
		\end{align*}
		and the proof is completed by noting that $\M_2(\eta)\equiv\eta(\V_2)$.
	\end{proof}

	\subsection{Propagation of Chaos}
	Now we can put the previous results together in a statement of the propagation of chaos for the Markov chain \eqref{eq:mv_mc}.
	\begin{theorem}
		Suppose that Assumption~\ref{assump:mv} is true. Then $\tau_1(K^\delta_\bullet)<\infty$ and is explicitly bounded by \eqref{eq:tau_mv}, and 
		if $\E[\|X^1_0\|^2]<\infty$, then there are constants $M^{(2)}_n$ such that
	 	$\E[\|X^1_n\|^2]\leq M^{(2)}_n$ for every $n\geq 0$ and therefore
	 	\[
	 		\W_1(\eta^{N,q}_n, \eta^{\otimes q}_n)\leq C_{2,d}\frac{q}{N^{1/d}} \sum_{k=0}^n \left(M^{(2)}_k\right)^{1/2}\tau_1(K^\delta_\bullet)^{n-k}.
	 	\]If, additionally, $\lambda_V,\lambda_W,C_V, C_W,~\delta$ are such that $\tau_1(K^\delta_\bullet)<1$ and $a + b<1$ from Proposition~\ref{prop:mv_mom}, in particular $\lambda_V + \lambda_W > 0$ and $V + W$ is sufficiently convex, then this propagation of chaos is uniform in $n$.
	\end{theorem}
	\begin{proof}
		This follows immediately from Theorem~\ref{thm:mv_lip}, Proposition~\ref{prop:mv_mom}, and Theorem~\ref{thm:poc_good} since $K^{\delta}_\bullet$ satisfies the required form of regularity in Assumption~\ref{assump:R2}.
	\end{proof}

    \section{Particle Filtering \& Feynman-Kac Distribution Flows}\label{sec:fkm}
	Finally, we study an important class of nonlinear Markov chains connected to many areas of applied mathematics including, but not limited to, nonlinear filtering problems \cite{del1997nonlinear}, sequential Monte Carlo \cite{del2006sequential}, genetic algorithms \cite{del2001modeling}, and Markov chain Monte Carlo \cite{andrieu2007non,vuckovic2022nonlinear}. This is the so-called Feynman-Kac model framework studied at length in \cite{del2004feynman} and related works. We do not go into full depth here as the topic is extensive, but focus on the nonlinear Markov chain interpretation and its connections to nonlinear filtering.
	
	\subsection{Problem Setup}
	Consider a family of nonlinear measure maps $\Phi_n:\calP(\R^d)\to \calP(\R^d)$ defined as 
	\begin{equation}\label{eq:fkm}
		\Phi_n[\eta] := \Psi_{G_{n-1}}(\eta)M^{(n)} = \int \Psi_{G_{n-1}}(\eta)(\dd x)M^{(n)}(x,\bullet)
	\end{equation}
	where $M^{(n)}$ is a family of time-inhomoegeous \emph{linear} Markov kernels, $G_n:\R^d\to [0,\infty)$ is a family of potential functions, and $\Psi_{G_n}(\eta)$ is the \emph{Boltzmann-Gibbs transformation} associated to $G_n$, defined for measurable $f:\R^d\to \R$ as 
	\begin{equation}\label{eq:bg}
		\Psi_{G_n}(\eta)(f) := \frac{1}{\eta(G_n)}\int f(x)G_n(x)\eta(\dd x)
	\end{equation}
	whenever $\eta(G_n)>0$.\nnewline
	
	It is possible to obtain a (nonunique) family of nonlinear Markov kernels $K^{(n)}_\bullet$ that satisfy $\Phi_n[\eta] = \eta K^{(n)}_{\eta}$; we will study the choice 
	\begin{align}
		K^{(n)}_\eta(x, \dd y)&:= S^{(n-1)}_\eta M^{(n)}(x, \dd y);\\
		S^{(n)}_\eta(x, \dd y)&:= \lambda_n G_n(x)\delta_x(\dd y) + (1-\lambda_n G_n(x))\Psi_{G_n}(\eta)(\dd y)
	\end{align}
	where $\lambda_n\in (0,1)$ is chosen such that $\lambda_n G_n(x) \leq 1~\forall x\in\R^d$. In particular, $\Psi_{G_n}(\eta) = \eta S^{(n)}_{\eta}$ for any $\eta\in \calP(\R^d)$ for which all terms are defined. This corresponds to ``selection-mutation'' dynamics wherein the potential functions $G_n$ determine the ``fitness'' of $x$, $S^{(n-1)}_\eta(x, \dd y)$ samples according to this fitness with rejection, and then $M^{(n)}$ mutates the sample $x$ independently.

	\paragraph{Connection to Bayesian Filtering}
	The system \eqref{eq:fkm} has a deep connection to Bayesian (aka, nonlinear) filtering. To be specific, consider a sequence of pairs of random variables $(X_n, Y_n)$ where $X_n$ is the \emph{state} variable which is unobserved, and $Y_n$ is the observation variable, which is observed. We are interested in computing the \emph{smoothed} distribution $\P(X_n\in \dd x_n | Y_{\leq n}=y_{\leq n})$ and the \emph{prediction} distribution $\P(X_{n+1}\in \dd x_{n+1}|Y_{\leq n}=y_{\leq n})$. To see how this relates to \eqref{eq:fkm}, first note the connection between the Boltzmann-Gibbs transformation \eqref{eq:bg} and Bayes' formula: for a likelihood density with a fixed $y$ written as $G(x) := \P(Y=y|X=x)$ and prior $\P(X\in \dd x)$ we have
	\begin{align*}
		\P(X\in \dd x | Y=y)&= \frac{\P(Y=y|X=x)}{\P(Y=y)}\P(X\in \dd x)\\
		&=\frac{\P(Y=y|X=x)}{\int \P(Y=y|X=x')\P(x\in \dd x')}\P(X\in \dd x)\\
		&=\frac{G(x)}{\int G(x')\P(X\in \dd x')}\P(X\in \dd x)\\
		&=\Psi_G(\P(X\in \bullet))(\dd x).
	\end{align*}
	Now, for a fixed realization $\{Y_n=y_n\}_{n=0}^\infty$, recall the Bayesian filtering equations: first, we have the smoothing equation
	\begin{align*}
		\P(X_n\in \dd x|Y_{\leq n}=y_{\leq n}) 
		&= \frac{\P(Y_n=y_n|X_n=x)}{\P(Y_n=y_n|Y_{<n}=y_{<n})}\P(X_n\in \dd x| Y_{<n}=y_{<n})\\
		&= \frac{\P(Y_n=y_n|X_n=x)}{\int \P(Y_n=y_n|X_n=x')\P(X_n\in \dd x'|Y_{<n}=y_{<n})}\P(X_n\in \dd x| Y_{<n}=y_{<n})
	\end{align*}
	which by the above remark can be written as
	\[
		\P(X_n\in \dd x|Y_{\leq n}=y_{\leq n}) 
		= \Psi_{G_n}(\P(X_n\in \bullet| Y_{<n}=y_{<n}))(\dd x)
	\]for the observation likelihood $G_n(x) := \P(Y_n=y_n | X_n=x)$. Then we have the prediction equation
	\[
		\P(X_{n+1}\in \dd x | Y_{\leq n}=y_{\leq n}) = \int \P(X_{n+1}\in \dd x|X_{n}=x')\P(X_{n}\in \dd x' | Y_{\leq n}=y_{\leq n})
	\]which can be written
	\[
		\P(X_{n+1}\in \dd x | Y_{\leq n}=y_{\leq n}) = \int \P(X_{n}\in \dd x' | Y_{\leq n}=y_{\leq n})M^{(n+1)}(x', \dd x)
	\]for
	\[
		M^{(n+1)}(x, \dd x'):=\P(X_{n+1}\in \dd x'|X_{n}=x).
	\]Hence, by identifying $\eta_n(\dd x):=\P(X_n\in \dd x|Y_{<n}=y_{<n})$, we have reformulated the Bayesian filtering equations as a Feyman-Kac distribution flow which is precisely \eqref{eq:fkm}.

	\paragraph{Simulation and Particle Filtering}
	Solving the nonlinear filtering problem is plainly valuable for many applications. However, except for special cases, the Markov chain $X_{n+1}\sim K^{(n+1)}_{\eta_n}(X_n,\bullet)$ cannot be simulated directly due to the dependence on $\eta_n$. On the other hand, one may approximate $\eta_n\approx \eta^N_n=:m(\ovl{X}_n)$ with an empirical measure, and obtain a family of algorithms for filtering generally called ``particle filtering'' algorithms that implement a time-inhomoheneous Markov chain
	\[
		X^i_n \sim K^{(n+1)}_{\eta^N_n}(X^i_{n},\bullet)
	\]which \emph{can} be simulated. In particular, this kernel models the popular ``SIS-R filter'' as a self-interacting Markov chain. Clearly, understanding the accuracy of the resulting particle system as a function of $N$ has practical implications on the accuracy of real-world decisions based on filtered estimates.

	\subsection{Results}

	\begin{assumption}\label{assump:G1}
		$0<\veps <G_n(x) \leq \ovl{G}_n < \infty~\forall x\in \R^d$.
	\end{assumption}
	\begin{remark}
		This is the setting in which many analyses of particle filtering methods take place. While it allows for convenient estimates of moments and regularity, it excludes many common cases such as Gaussian observation densities, which are handled separately. Nevertheless, for the scope of the current work, we use this common setting and derive quantitative propagation of chaos.
	\end{remark}

	\begin{theorem}[Propagation of Chaos, Informal]
		Under sufficient assumptions on $G_n$ and $M^{(n)}$ including Assumption~\ref{assump:G1}, we have
        \[
        	\W_1(\eta^{N,q}_n,\eta^{\otimes q}_n)\leq  C_{p, d}\frac{q}{N^{1/d}}\sum_{k=0}^n \left(M^{(p)}_k\right)^{1/p}\prod_{j=k+1}^n \tau_1(\Phi_j)[\eta_{j-1}]  + 2\frac{q^3}{N}(M^{(p)}_n)^{1/p}.
        \]
	\end{theorem}

	The plan for proving this result is the same as in Section~\ref{sec:mv}: establish moment bounds, estimate $\tau_1(K^{(n)}_\bullet)$, and then combine everything together into a propagation of chaos result.
	\paragraph{Moment Bounds}
	\begin{assumption}\label{assump:Mmom}
		Let $M^{(n)}$ be the Markov kernel from above. There exists $\wtilde{a}_n,c\in (0,\infty)$ such that
		\[
			M^{(n)} V(x)\leq \wtilde{a}_nV(x) + c
		\]for some $V:\R^d\to[0,\infty)$ satisfying $\|x\|^p\leq \theta V(x)~\forall \|x\|>R$ with $p>1,\theta>0$, and $R\geq 0$. 
	\end{assumption}

	\begin{proposition}\label{prop:kf_mom_lb}
		Suppose that Assumptions~\ref{assump:G1},~\ref{assump:Mmom} hold; then
		\[
			K^{(n)}_\eta V(x)\leq a_n V(x) + b_n\eta(V) + c
		\]for every $n\geq 1$ and $\eta(V)<\infty$, with
		\begin{align*}
			a_n&:= \wtilde{a}_n \lambda_{n-1} \ovl{G}_{n-1}\\
			b_n&:= \wtilde{a}_n\ovl{G}_{n-1}(1-\lambda_{n-1} \veps)\veps^{-1}
		\end{align*}
		and $c$ from Assumption~\ref{assump:Mmom}.
	\end{proposition}
	\begin{proof}
		Recall that 
		\[
			K^{(n)}_\eta V(x) = \int S^{(n-1)}_\eta(x, \dd y)M^{(n)}(y, \dd z)V(z) = \int S^{(n-1)}_\eta(x, \dd y)M^{(n)}V(y)
		\]But by assumption $M^{(n)}V(y)\leq \wtilde{a}_nV(y) + c$ hence
		\[
			K^{(n)}_\eta V(x)\leq \wtilde{a}_n\int S^{(n-1)}_\eta(x, \dd y) V(y) + c.
		\]Hence consider
		\begin{align*}
			S^{(n-1)}_\eta V(x) &=\lambda_{n-1} G_{n-1}(x)\int \delta_x(\dd y)V(y) + (1 - \lambda_{n-1} G_{n-1}(x))\Psi_{G_{n-1}}(\eta)(V)\\
			&=\lambda_{n-1} G_{n-1}(x)V(x) + (1-\lambda_{n-1} G_{n-1}(x))\Psi_{G_{n-1}}(\eta)(V)\\
            &\leq \lambda_{n-1} \ovl{G}_{n-1} V(x) + (1-\lambda_{n-1} \veps)\Psi_{G_{n-1}}(\eta)(V)\\
            &= \lambda_{n-1} \ovl{G}_{n-1} V(x) +  (1-\lambda_{n-1} \veps)\frac{\eta(G_{n-1}V)}{\eta(G_{n-1})}\\
            &\leq \lambda_{n-1} \ovl{G}_{n-1} V(x) + \ovl{G}_{n-1}(1-\lambda_{n-1} \veps)\veps^{-1}\eta(V).
		\end{align*}
		Therefore
		\[
			K^{(n)}_\eta V(x) \leq a_n V(x) + b_n\eta(V) + c.
		\]
	\end{proof}

	\begin{corollary}\label{coro:fk_mom}
		Using the notation from Proposition~\ref{prop:kf_mom_lb}, we have that
		\[
			a_n + b_n =\wtilde{a}_n \frac{\ovl{G}_{n-1}}{\veps}
		\]so in particular, if $\wtilde{a}_n \frac{\ovl{G}_{n-1}}{\veps}<1~\forall n\geq 1$ and $\E[\M_p(\eta^N_0)]<\infty$ then $\E[\M_p(\Phi_n[\eta^N_{n-1}])]$ is uniformly bounded.
	\end{corollary}
	\begin{proof}
		This is simple algebra.
		\begin{align*}
			a_n + b_n &= \wtilde{a}_n \lambda_{n-1} \ovl{G}_{n-1} + \wtilde{a}_n\ovl{G}_{n-1}(1-\lambda_{n-1} \veps)\veps^{-1}\\
			&=\wtilde{a}_n\left(\frac{\veps\lambda_{n-1} \ovl{G}_{n-1}}{\veps} + \frac{\ovl{G}_{n-1}(1-\lambda_n\eps)}{\eps}\right)\\
			&= \wtilde{a}_n \frac{\ovl{G}_{n-1}}{\veps}.
		\end{align*}
	\end{proof}
	\begin{remark}
		Let us comment on the intuition of $\wtilde{a}_n\ovl{G}_{n-1}/\veps < 1$. Another way of describing $\ovl{G}_{n-1}/\veps$ is
		\[
			\frac{\ovl{G}_{n-1}}{\veps}=\frac{\max_x G_{n-1}(x)}{\min_y G_{n-1}(y)} = \max_{x,y}\frac{G_{n-1}(x)}{G_{n-1}(y)}.
		\]This is equivalent to the existence of a number $\eps(G_{n-1})>0$ such that $G_{n-1}(y)\geq \eps(G_{n-1})G_{n-1}(x)$ which is in fact the ubiquitous ``condition $(G)$'' from \cite[Section~3.5.2]{del2004feynman} and measures the ``mixing'' of $G$. Therefore, one can interpret $\wtilde{a}_n\ovl{G}_{n-1}/\veps < 1$ as balancing the moment containment of $M^{(n)}$ against the mixing of $G_{n-1}$ --- in particular, the less mixing $G_{n-1}$ is, the more $M^{(n)}$ must compensate to maintain bounded moments.
	\end{remark}
	
	\paragraph{Lipschitz Regularity}
	First, we work with a general $G:\R^d\to (0,\infty)$ when estimating regularity.
	\begin{assumption}[Regularity of $G$]\label{assump:Greg}
		\hfill
		\begin{assumptions}[G]
			\item $\|G\|_{\lip}<\infty$.
			\item $G$ is $C^1$ and
			 \[
				\sup_{z\in \R^d}\|z\|\cdot\|\nabla G(z)\|\leq \wtilde{G}<\infty.
			\]
		\end{assumptions}
	\end{assumption}

	\begin{lemma}\label{lem:psi_reg}
		Assume that $\ovl{G}:=\|G\|_\infty< \infty$ and Assumption~\ref{assump:Greg} holds. Then for any $\mu,\nu\in\calP_1(\R^d)$ for which $\mu(G),\nu(G)>0$,
		\[
			\W_1(\Psi_G(\mu), \Psi_G(\nu))\leq \left(\frac{\|G\|_{\lip} + \wtilde{G} + \ovl{G}}{\mu(G)} + \frac{\|G\|_{\lip}\ovl{G}}{\mu(G)\nu(G)}\M_1(\nu)\right)\W_1(\mu,\nu).
		\]
	\end{lemma}
	\begin{proof}
		Let $f\in \lip(1)$. Then
		\begin{align*}
			|\Psi_G(\mu)(f) - \Psi_G(\nu)(f)|&=\left|\frac{\mu(Gf)}{\mu(G)} - \frac{\nu(Gf)}{\nu(G)}\right|\\
			&\leq \frac{1}{\mu(G)\nu(G)}|\mu(Gf)\nu(G) - \nu(Gf)\mu(G)|\\
			&\leq \underbrace{\frac{\nu(G)}{\mu(G)\nu(G)}|\mu(Gf) - \nu(Gf)|}_{(1)} + \underbrace{\frac{|\nu(Gf)|}{\mu(G)\nu(G)}|\nu(G) - \mu(G)|}_{(2)}.
		\end{align*}
		For the second term
		\begin{align*}
			(2)&\leq \frac{|\nu(Gf)|}{\mu(G)\nu(G)}\|G\|_{\lip}\W_1(\mu,\nu)
		\end{align*}
		and without loss of generality we may assume $f(0)=0$ hence
		\[
			|\nu(Gf)|= \left|\int G(x)f(x)\nu(\dd x)\right|\leq \int G(x)\|x\|\nu(\dd x)\leq \ovl{G} M_1(\nu).
		\]Thus 
		\[
			(2)\leq \frac{\ovl{G} M_1(\nu)}{\mu(G)\nu(G)}\|G\|_{\lip}\W_1(\mu,\nu).
		\]For the first term, we have
		\[
			(1)\leq \frac{1}{\mu(G)}\|Gf\|_{\lip}\W_1(\mu,\nu)
		\]so now we must estimate $\|Gf\|_{\lip}$. But
		\begin{align*}
		    \|Gf\|_{\lip} = \sup_{x\neq y}\frac{|G(x)f(x) - G(y)f(y)|}{\|x-y\|}\\
		\end{align*}
		and without loss of generality, $\|x-y\|\leq 1$ therefore
		\begin{align*}
		    |G(x)f(x) - G(y)f(y)| &\leq |f(x)||G(x) - G(y)| + |G(y)||f(x) - f(y)|\\
		    &\leq |f(x) - f(0)||G(x) - G(y)| + |G(y)||f(x) - f(y)|\\
		    &\leq |f(x) - f(0)||G(x) - G(y)| + \ovl{G}\|x-y\|.
		\end{align*}
		For the first term, using the Mean Value Theorem we get
		\begin{align*}
			|f(x)||G(x) - G(y)|&\leq |f(x)||\|\nabla G(tx + (1-t)y)\|\|x-y\|
		\end{align*}
		for $t\in [0,1]$ so that
		\begin{align*}
			|f(x)|\|\nabla G(tx + (1-t)y)\|&\leq |f(x) - f(0)|\|\nabla G((1-t)x + ty)\|\\
			&\leq \|x\|\|\nabla G((1-t)x + ty)\|\\
			&\leq (\|(1-t)x + ty\| + t\|x-y\|)\|\nabla G((1-t)x + ty)\|\\
			&\leq (1 + \|(1-t)x + ty\|)\|\nabla G((1-t)x + ty)\|\\
			&\leq \|\nabla G\|_\infty + \wtilde{G}.
		\end{align*}
		Recalling that $\|\nabla G\|_\infty\leq \|G\|_{\lip}$ we have
		\[
			\|Gf\|_{\lip}\leq \|G\|_{\lip} + \wtilde{G} + \ovl{G}
		\]and therefore
		\[
			(1)\leq \frac{\|G\|_{\lip} + \wtilde{G} + \ovl{G}}{\mu(G)}\W_1(\mu,\nu).
		\]Putting everything together yields the result.
	\end{proof}
	\begin{lemma}\label{lem:phi_reg}
		Suppose Assumptions~\ref{assump:G1},\ref{assump:Greg} hold. 
		Then
		\[
			\tau_1(\Phi_{k})[\mu,\nu]\leq \tau_1(\Psi_{G_{k-1}})[\mu,\nu]\cdot\tau_1(M^{(k)})	
		\]
		with
		\[
			\tau_1(\Psi_{G_{n}})[\mu,\nu]\leq \frac{\|\nabla G_n\|_\infty + \wtilde{G}_n + \ovl{G}_n}{\mu(G_n)} + \frac{\|G_n\|_{\lip}\ovl{G}_n}{\mu(G_n)\nu(G_n)}\M_1(\nu).
		\]
	\end{lemma}
	\begin{proof}
        Let $\mu,\nu\in\calP_1(\R^d)$ and $\mu\neq \nu$. We have
        \begin{align*}
        	\W_1(\Phi_{n}[\mu], \Phi_n[\nu]) &= \W_1(\Psi_{G_{n-1}}(\mu)M^{(n)}, \Psi_{G_{n-1}}(\nu)M^{(n)})\\
        	&= \frac{\W_1(\Psi_{G_{n-1}}(\mu)M^{(n)}, \Psi_{G_{n-1}}(\nu)M^{(n)})}{\W_1(\Psi_{G_{n-1}}(\mu), \Psi_{G_{n-1}}(\nu))}\W_1(\Psi_{G_{n-1}}(\mu), \Psi_{G_{n-1}}(\nu))\\
        	&\leq \tau_1(M^{(n)})\W_1(\Psi_{G_{n-1}}(\mu), \Psi_{G_{n-1}}(\nu))\\
        	&\leq \tau_1(M^{(n)})\tau_1(\Psi_{G_{n-1}})[\mu,\nu]\W_1(\mu,\nu).
        \end{align*}
	\end{proof}

	\begin{corollary}\label{coro:phi_reg}
		Suppose that Assumptions~\ref{assump:G1},\ref{assump:Greg} hold. Then
		\[
			\tau_1(\Phi_{k})[\mu,\nu] \leq \tau_1(\Phi_k)[\nu]:=\tau_1(\Psi_{G_{k-1}})[\nu]\tau_1(M^{(k)})
		\]where Lemma~\ref{lem:phi_reg} applies with
		\[
			\tau_1(\Psi_{G_{n}})[\nu]:=\frac{\|\nabla G_n\|_\infty + \wtilde{G}_n + \ovl{G}_n}{\eps} + \frac{\|G_n\|_{\lip}\ovl{G}_n}{\eps^2}\M_1(\nu).
		\]
	\end{corollary}

	\paragraph{Propagation of Chaos}
	\begin{theorem}\label{thm:fk_poc}
		Suppose that Assumptions~\ref{assump:G1},\ref{assump:Mmom},\ref{assump:Greg} hold. Then
        \[
        	\W_1(\eta^{N,q}_n,\eta^{\otimes q}_n)\leq  C_{p, d}\frac{q}{N^{1/d}}\sum_{k=0}^n \left(M^{(p)}_k\right)^{1/p}\prod_{j=k+1}^n \tau_1(\Phi_j)[\eta_{j-1}]  + 2\frac{q^3}{N}(M^{(p)}_n)^{1/p}.
        \] 
	\end{theorem}
	\begin{proof}
		This follows immediately from Theorem~\ref{thm:poc_general} with the one-sided Lipschitz estimate from Proposition~\ref{prop:emp_main_oneside} since $K^{(n)}_\bullet$ satisfies the correct one-sided local Lipschitz estimate described therein.
	\end{proof}

	\bibliography{regularity_poc.bib}

@book{del2004feynman,
  title={{F}eynman-{K}ac formulae},
  author={Del Moral, Pierre},
  year={2004},
  publisher={Springer}
}

@article{fournier2015rate,
  title={On the rate of convergence in Wasserstein distance of the empirical measure},
  author={Fournier, Nicolas and Guillin, Arnaud},
  journal={Probability theory and related fields},
  volume={162},
  number={3},
  pages={707--738},
  year={2015},
  publisher={Springer}
}

@article{clarte2022collective,
  title={Collective proposal distributions for nonlinear {MCMC} samplers: Mean-field theory and fast implementation},
  author={Clart{\'e}, Gr{\'e}goire and Diez, Antoine and Feydy, Jean},
  journal={Electronic Journal of Statistics},
  volume={16},
  number={2},
  pages={6395--6460},
  year={2022},
  publisher={The Institute of Mathematical Statistics and the Bernoulli Society}
}

@article{guillin2021uniform,
  title={Uniform long-time and propagation of chaos estimates for mean field kinetic particles in non-convex landscapes},
  author={Guillin, Arnaud and Monmarch{\'e}, Pierre},
  journal={Journal of Statistical Physics},
  volume={185},
  pages={1--20},
  year={2021},
  publisher={Springer}
}

@article{durmus2020elementary,
  title={An elementary approach to uniform in time propagation of chaos},
  author={Durmus, Alain and Eberle, Andreas and Guillin, Arnaud and Zimmer, Raphael},
  journal={Proceedings of the American Mathematical Society},
  volume={148},
  number={12},
  pages={5387--5398},
  year={2020}
}

@article{carrillo2003kinetic,
  title={Kinetic equilibration rates for granular media and related equations: entropy dissipation and mass transportation estimates},
  author={Carrillo, Jos{\'e} A and McCann, Robert J and Villani, C{\'e}dric},
  journal={Revista Matematica Iberoamericana},
  volume={19},
  number={3},
  pages={971--1018},
  year={2003}
}

@article{malrieu2003convergence,
  title={Convergence to equilibrium for granular media equations and their Euler schemes},
  author={Malrieu, Florent},
  journal={The Annals of Applied Probability},
  volume={13},
  number={2},
  pages={540--560},
  year={2003},
  publisher={Institute of Mathematical Statistics}
}

@article{sznitman1991topics,
  title={Topics in propagation of chaos},
  author={Sznitman, Alain-Sol},
  journal={Ecole d’{\'e}t{\'e} de probabilit{\'e}s de Saint-Flour XIX—1989},
  volume={1464},
  pages={165--251},
  year={1991},
  publisher={Springer}
}

@article{vuckovic2022nonlinear,
  title={Nonlinear {MCMC} for {B}ayesian machine learning},
  author={Vuckovic, James},
  journal={Advances in Neural Information Processing Systems},
  volume={35},
  pages={28400--28413},
  year={2022}
}

@article{del1997nonlinear,
  title={Nonlinear filtering: Interacting particle resolution},
  author={Del Moral, Pierre},
  journal={Comptes Rendus de l'Acad{\'e}mie des Sciences-Series I-Mathematics},
  volume={325},
  number={6},
  pages={653--658},
  year={1997},
  publisher={Elsevier}
}

@article{del2006sequential,
  title={Sequential {M}onte {C}arlo samplers},
  author={Del Moral, Pierre and Doucet, Arnaud and Jasra, Ajay},
  journal={Journal of the Royal Statistical Society Series B: Statistical Methodology},
  volume={68},
  number={3},
  pages={411--436},
  year={2006},
  publisher={Oxford University Press}
}

@inproceedings{andrieu2007non,
  title={Non-linear {M}arkov chain {M}onte {C}arlo},
  author={Andrieu, Christophe and Jasra, Ajay and Doucet, Arnaud and Del Moral, Pierre},
  booktitle={Esaim: Proceedings},
  volume={19},
  pages={79--84},
  year={2007},
  organization={EDP Sciences}
}

@article{del2001modeling,
  title={Modeling genetic algorithms with interacting particle systems},
  author={Del Moral, P and Kallel, Leila and Rowe, Jonathan},
  journal={Revista de Matem{\'a}tica: Teor{\'\i}a y Aplicaciones},
  volume={8},
  number={2},
  pages={19--77},
  year={2001}
}

@article{fournier2023convergence,
  title={Convergence of the empirical measure in expected wasserstein distance: non-asymptotic explicit bounds in Rd},
  author={Fournier, Nicolas},
  journal={ESAIM: Probability and Statistics},
  volume={27},
  pages={749--775},
  year={2023},
  publisher={EDP Sciences}
}

@article{gozlan2009characterization,
  title={A Characterization Of Dimension Free Concentration In Terms Of Transportation Inequalities},
  author={Gozlan, Nathael},
  journal={The Annals of Probability},
  volume={37},
  number={6},
  pages={2480--2498},
  year={2009}
}

@article{bolley2007quantitative,
  title={Quantitative concentration inequalities for empirical measures on non-compact spaces},
  author={Bolley, Fran{\c{c}}ois and Guillin, Arnaud and Villani, C{\'e}dric},
  journal={Probability Theory and Related Fields},
  volume={137},
  pages={541--593},
  year={2007},
  publisher={Springer}
}

@article{del2007sharp,
  title={Sharp propagation of chaos estimates for {F}eynman--{K}ac particle models},
  author={Del Moral, Pierre and Doucet, Arnaud and Peters, Gareth W},
  journal={Theory of Probability \& Its Applications},
  volume={51},
  number={3},
  pages={459--485},
  year={2007},
  publisher={SIAM}
}

@misc{chafai2020coupling,
  author={Djalil Chafaï},
  title={Coupling, divergences, and {M}arkov kernels},
  year={2020},
  url={https://djalil.chafai.net/blog/2020/03/15/coupling-divergences-and-markov-kernels/},
  note={Accessed: 2025-04-26}
}

@article{mckean1966class,
  title={A class of {M}arkov processes associated with nonlinear parabolic equations},
  author={{M}cKean Jr, Henry P},
  journal={Proceedings of the National Academy of Sciences},
  volume={56},
  number={6},
  pages={1907--1911},
  year={1966}
}

@article{guillin2022convergence,
  title={Convergence rates for the {V}lasov-{F}okker-{P}lanck equation and uniform in time propagation of chaos in non convex cases},
  author={Guillin, Arnaud and Le Bris, Pierre and Monmarch{\'e}, Pierre},
  journal={Electronic Journal of Probability},
  volume={27},
  pages={1--44},
  year={2022},
  publisher={The Institute of Mathematical Statistics and the Bernoulli Society}
}

@article{mischler2015new,
  title={A new approach to quantitative propagation of chaos for drift, diffusion and jump processes},
  author={Mischler, St{\'e}phane and Mouhot, Cl{\'e}ment and Wennberg, Bernt},
  journal={Probability Theory and Related Fields},
  volume={161},
  pages={1--59},
  year={2015},
  publisher={Springer}
}

@article{mischler2013kac,
  title={{K}ac’s program in kinetic theory},
  author={Mischler, St{\'e}phane and Mouhot, Cl{\'e}ment},
  journal={Inventiones mathematicae},
  volume={193},
  pages={1--147},
  year={2013},
  publisher={Springer}
}

@book{kolokoltsov2010nonlinear,
  title={Nonlinear {M}arkov processes and kinetic equations},
  author={Kolokoltsov, Vassili N},
  volume={182},
  year={2010},
  publisher={Cambridge University Press}
}

@article{guillin2022uniform,
  title={Uniform {P}oincar{\'e} and logarithmic {S}obolev inequalities for mean field particle systems},
  author={Guillin, Arnaud and Liu, Wei and Wu, Liming and Zhang, Chaoen},
  journal={The Annals of Applied Probability},
  volume={32},
  number={3},
  pages={1590--1614},
  year={2022},
  publisher={Institute of Mathematical Statistics}
}

@article{jabin2018quantitative,
  title={Quantitative estimates of propagation of chaos for stochastic systems with W-1,infty kernels},
  author={Jabin, Pierre-Emmanuel and Wang, Zhenfu},
  journal={Inventiones mathematicae},
  volume={214},
  pages={523--591},
  year={2018},
  publisher={Springer}
}

@article{lacker2023hierarchies,
  title={Hierarchies, entropy, and quantitative propagation of chaos for mean field diffusions},
  author={Lacker, Daniel},
  journal={Probability and Mathematical Physics},
  volume={4},
  number={2},
  pages={377--432},
  year={2023},
  publisher={Mathematical Sciences Publishers}
}

@article{lacker2023sharp,
  title={Sharp uniform-in-time propagation of chaos},
  author={Lacker, Daniel and Le Flem, Luc},
  journal={Probability Theory and Related Fields},
  volume={187},
  number={1-2},
  pages={443--480},
  year={2023},
  publisher={Springer}
}

@article{lasry2007mean,
  title={Mean field games},
  author={Lasry, Jean-Michel and Lions, Pierre-Louis},
  journal={Japanese journal of mathematics},
  volume={2},
  number={1},
  pages={229--260},
  year={2007},
  publisher={Springer}
}

@article{caines2006large,
  title={Large population stochastic dynamic games: closed-loop {M}cKean-{V}lasov systems and the Nash certainty equivalence principle},
  author={Caines, Peter E and Huang, Minyi and Malham{\'e}, Roland P},
  journal={Communications in Information and Systems},
  volume={6},
  number={3},
  pages={221--252},
  year={2006},
  publisher={International Press}
}

@article{del1998measure,
  title={Measure-valued processes and interacting particle systems. Application to nonlinear filtering problems},
  author={Del Moral, Pierre},
  journal={The Annals of Applied Probability},
  volume={8},
  number={2},
  pages={438--495},
  year={1998},
  publisher={Institute of Mathematical Statistics}
}

@article{gordon1995bayesian,
  title={{B}ayesian state estimation for tracking and guidance using the bootstrap filter},
  author={Gordon, Neil and Salmond, David and Ewing, Craig},
  journal={Journal of Guidance, Control, and Dynamics},
  volume={18},
  number={6},
  pages={1434--1443},
  year={1995}
}

@article{kitagawa1996monte,
  title={{M}onte {C}arlo filter and smoother for non-{G}aussian nonlinear state space models},
  author={Kitagawa, Genshiro},
  journal={Journal of computational and graphical statistics},
  volume={5},
  number={1},
  pages={1--25},
  year={1996},
  publisher={Taylor \& Francis}
}

@article{carvalho1997optimal,
  title={Optimal nonlinear filtering in GPS/INS integration},
  author={Carvalho, Himilcon and Del Moral, Pierre and Monin, Andr{\'e} and Salut, G{\'e}rard},
  journal={IEEE Transactions on Aerospace and Electronic Systems},
  volume={33},
  number={3},
  pages={835--850},
  year={1997},
  publisher={IEEE}
}

@article{del1999central,
  title={Central limit theorem for nonlinear filtering and interacting particle systems},
  author={Del Moral, Pierre and Guionnet, Alice},
  journal={Annals of Applied Probability},
  pages={275--297},
  year={1999},
  publisher={JSTOR}
}

@article{del2012adaptive,
  title={An adaptive sequential {M}onte {C}arlo method for approximate {B}ayesian computation},
  author={Del Moral, Pierre and Doucet, Arnaud and Jasra, Ajay},
  journal={Statistics and computing},
  volume={22},
  pages={1009--1020},
  year={2012},
  publisher={Springer}
}

@article{doi:10.1137/S0040585X97986825,
  author = {Butkovsky, O. A.},
  title = {On Ergodic Properties of Nonlinear {M}arkov Chains and Stochastic {M}cKean--{V}lasov Equations},
  journal = {Theory of Probability \& Its Applications},
  volume = {58},
  number = {4},
  pages = {661-674},
  year = {2014},
  doi = {10.1137/S0040585X97986825},
  URL = {https://doi.org/10.1137/S0040585X97986825},
  eprint = {https://doi.org/10.1137/S0040585X97986825}
}

@article{moral2010interacting,
  title={Interacting {M}arkov Chain {M}onte {C}arlo Methods For Solving Nonlinear Measure-Valued Equations},
  author={Moral, Pierre Del and Doucet, Arnaud},
  journal={The Annals of Applied Probability},
  pages={593--639},
  year={2010},
  publisher={JSTOR}
}

@article{andrieu2010particle,
  title={Particle {M}arkov chain {M}onte {C}arlo methods},
  author={Andrieu, Christophe and Doucet, Arnaud and Holenstein, Roman},
  journal={Journal of the Royal Statistical Society Series B: Statistical Methodology},
  volume={72},
  number={3},
  pages={269--342},
  year={2010},
  publisher={Oxford University Press}
}

@inproceedings{rainforth2016interacting,
  title={Interacting particle {M}arkov chain {M}onte {C}arlo},
  author={Rainforth, Tom and Naesseth, Christian and Lindsten, Fredrik and Paige, Brooks and Vandemeent, Jan-Willem and Doucet, Arnaud and Wood, Frank},
  booktitle={International Conference on Machine Learning},
  pages={2616--2625},
  year={2016},
  organization={PMLR}
}

@book{kac1959probability,
  title={Probability and related topics in physical sciences},
  author={{K}ac, Mark},
  volume={1},
  year={1959},
  publisher={American Mathematical Soc.}
}

@article{carrillo2006contractions,
  title={Contractions in the 2-Wasserstein length space and thermalization of granular media},
  author={Carrillo, Jos{\'e} A and McCann, Robert J and Villani, C{\'e}dric},
  journal={Archive for Rational Mechanics and Analysis},
  volume={179},
  pages={217--263},
  year={2006},
  publisher={Springer}
}

@article{bakry1997sobolev,
  title={On {S}obolev and logarithmic {S}obolev inequalities for {M}arkov semigroups},
  author={Bakry, Dominique},
  journal={New trends in stochastic analysis (Charingworth, 1994)},
  pages={43--75},
  year={1997},
  publisher={World Scientific}
}

@article{eberle2019quantitative,
  title={Quantitative Harris-type theorems for diffusions and {M}cKean--{V}lasov processes},
  author={Eberle, Andreas and Guillin, Arnaud and Zimmer, Raphael},
  journal={Transactions of the American Mathematical Society},
  volume={371},
  number={10},
  pages={7135--7173},
  year={2019}
}

@inproceedings{del2001stability,
  title={On the stability of interacting processes with applications to filtering and genetic algorithms},
  author={Del Moral, Pierre and Guionnet, Alice},
  booktitle={Annales de l'Institut Henri Poincar{\'e} (B) Probability and Statistics},
  volume={37},
  number={2},
  pages={155--194},
  year={2001},
  organization={Elsevier}
}

@article{djuric2003particle,
  title={Particle filtering},
  author={Djuric, Petar M and Kotecha, Jayesh H and Zhang, Jianqui and Huang, Yufei and Ghirmai, Tadesse and Bugallo, M{\'o}nica F and Miguez, Joaquin},
  journal={IEEE signal processing magazine},
  volume={20},
  number={5},
  pages={19--38},
  year={2003},
  publisher={IEEE}
}

@article{metropolis1953equation,
  title={Equation of state calculations by fast computing machines},
  author={Metropolis, Nicholas and Rosenbluth, Arianna W and Rosenbluth, Marshall N and Teller, Augusta H and Teller, Edward},
  journal={The journal of chemical physics},
  volume={21},
  number={6},
  pages={1087--1092},
  year={1953},
  publisher={American Institute of Physics}
}

@article{metropolis1949monte,
  title={The {M}onte {C}arlo method},
  author={Metropolis, Nicholas and Ulam, Stanislaw},
  journal={Journal of the American statistical association},
  volume={44},
  number={247},
  pages={335--341},
  year={1949},
  publisher={Taylor \& Francis}
}

@article{hastings1970monte,
  title={{M}onte {C}arlo sampling methods using {M}arkov chains and their applications},
  author={Hastings, WK},
  journal={Biometrika},
  volume={57},
  number={1},
  pages={97--109},
  year={1970}
}

@article{shchegolev2022convergence,
  title={On convergence rate bounds for a class of nonlinear {M}arkov chains},
  author={Shchegolev, Alexander and Veretennikov, Alexander},
  journal={arXiv preprint arXiv:2209.12834},
  year={2022}
}

@article{shchegolev2022new,
  title={A new rate of convergence estimate for homogeneous discrete-time nonlinear {M}arkov chains},
  author={Shchegolev, Aleksandr A},
  journal={Random Operators and Stochastic Equations},
  volume={30},
  number={3},
  pages={205--213},
  year={2022},
  publisher={De Gruyter}
}

@article{xu2022estimation,
  title={Estimation and Application of the Convergence Bounds for Nonlinear {M}arkov Chains},
  author={Xu, Kaichen},
  journal={arXiv preprint arXiv:2212.05304},
  year={2022}
}

@article{cattiaux2008probabilistic,
  title={Probabilistic approach for granular media equations in the non-uniformly convex case},
  author={Cattiaux, Patrick and Guillin, Arnaud and Malrieu, Florent},
  journal={Probability theory and related fields},
  volume={140},
  pages={19--40},
  year={2008},
  publisher={Springer}
}

@article{hauray2014kac,
  title={On {K}ac's chaos and related problems},
  author={Hauray, Maxime and Mischler, St{\'e}phane},
  journal={Journal of Functional Analysis},
  volume={266},
  number={10},
  pages={6055--6157},
  year={2014},
  publisher={Elsevier}
}

@article{bossy1997stochastic,
  title={A stochastic particle method for the {M}cKean-{V}lasov and the Burgers equation},
  author={Bossy, Mireille and Talay, Denis},
  journal={Mathematics of computation},
  volume={66},
  number={217},
  pages={157--192},
  year={1997}
}

@inproceedings{gordon1993novel,
  title={Novel approach to nonlinear/non-{G}aussian {B}ayesian state estimation},
  author={Gordon, Neil J and Salmond, David J and Smith, Adrian FM},
  booktitle={IEE proceedings F (radar and signal processing)},
  volume={140},
  number={2},
  pages={107--113},
  year={1993},
  organization={IET}
}

@article{handschin1970monte,
  title={{M}onte {C}arlo techniques for prediction and filtering of non-linear stochastic processes},
  author={Handschin, JE},
  journal={Automatica},
  volume={6},
  number={4},
  pages={555--563},
  year={1970},
  publisher={Elsevier}
}

@article{freed1981polymers,
  title={Polymers as self-avoiding walks},
  author={Freed, Karl F},
  journal={The Annals of Probability},
  pages={537--554},
  year={1981},
  publisher={JSTOR}
}

@book{doucet2013sequential,
  title={Sequential {M}onte {C}arlo Methods in Practice},
  author={Doucet, A. and Smith, A. and de Freitas, N. and Gordon, N.},
  isbn={9781475734379},
  series={Information Science and Statistics},
  url={https://books.google.ca/books?id=BWPaBwAAQBAJ},
  year={2013},
  publisher={Springer New York}
}

@inproceedings{bolley2005weighted,
  title={Weighted {C}sisz{\'a}r-{K}ullback-{P}insker inequalities and applications to transportation inequalities},
  author={Bolley, Fran{\c{c}}ois and Villani, C{\'e}dric},
  booktitle={Annales de la Facult{\'e} des sciences de Toulouse: Math{\'e}matiques},
  volume={14},
  number={3},
  pages={331--352},
  year={2005}
}

@book{graf2000foundations,
  title={Foundations of quantization for probability distributions},
  author={Graf, Siegfried and Luschgy, Harald},
  year={2000},
  publisher={Springer Science \& Business Media}
}

@inproceedings{hairer2011yet,
  title={Yet another look at {H}arris’ ergodic theorem for {M}arkov chains},
  author={Hairer, Martin and Mattingly, Jonathan C},
  booktitle={Seminar on Stochastic Analysis, Random Fields and Applications VI: Centro Stefano Franscini, Ascona, May 2008},
  pages={109--117},
  year={2011},
  organization={Springer}
}

@article{eberle2016reflection,
  title={Reflection couplings and contraction rates for diffusions},
  author={Eberle, Andreas},
  journal={Probability theory and related fields},
  volume={166},
  number={3},
  pages={851--886},
  year={2016},
  publisher={Springer}
}
	\newpage
	\appendix
	
    \section{Auxiliary Results}\label{app:auxiliary_results}
    \subsection{Properties of the Wasserstein Distance}

    \begin{proposition}[Basic Properties]
        \hfill
        \begin{enumerate}
            \item Let $K, Q$ be a Markov kernels on $\R^d$ s.t. $K(x, \bullet), Q(x, \bullet)\in \calP_1(\R^d)~\forall x\in \R^d$ and $\mu\in\calP_1(\R^d)$. Then
            \[
                \W_1(\mu K, \mu Q)\leq \int \W_1(K(x, \bullet), Q(x,\bullet))\mu(\dd x)
            \]
            \item Let $K$ be a Markov kernel s.t. $K(x, \bullet)\in \calP_p(\R^d)$, and $\mu,\nu\in\calP(\R^d)$. Then 
            \[
                \W_p^p(\mu K, \nu K)\leq \inf_{\gamma\in\calC(\mu,\nu)}\int \W_p^p(K(x, \bullet), K(y, \bullet))\gamma(\dd x, \dd y)
            \]
        \end{enumerate}
    \end{proposition}
    \begin{proof}
        \hfill
        \begin{enumerate}
            \item \begin{align*}
                \W_1(\mu K, \mu Q) &=\sup_{f\in \lip(1)}|\mu Kf - \mu Qf|\\
                &\leq \sup_{f\in \lip(1)}\int \mu(\dd x)|Kf(x) - Qf(x)|\\
                &\leq \int \mu(\dd x)\W_1(K(x,\bullet), Q(x, \bullet)).
            \end{align*}
            \item Following the argument from \cite{chafai2020coupling}, let $\pi_{x,y}$ be a coupling of $K(x, \bullet),K(y,\bullet)$ for every $x,y\in\R^d$ and let $\gamma$ be a coupling of $\mu,\nu$. Then $\pi(\dd x, \dd y):=\int \pi_{x,y}\gamma(\dd x, \dd y)$ is a coupling of $\mu K, \nu K$: 
            \begin{align*}
                \int_{y'} \pi(\dd x', \dd y') &= \int_{y'} \int_{x,y}\pi_{x,y}(\dd x', \dd y')\gamma(\dd x, \dd y)\\
                &= \int_{x,y}K(x, \dd x')\gamma(\dd x, \dd y)\\
                &= \int_{x}K(x, \dd x')\mu(\dd x)\\
                &= \mu K
            \end{align*}
            and similarly for the other marginal. Then, by standard measurable selection arguments, let $\pi_{x,y}$ be a $\W_p$-optimal coupling for $K(x,\bullet), K(y,\bullet)$. Then we have
            \begin{align*}
                \W_p^p(\mu K, \nu K)&\leq \int \|x' - y'\|^p \pi(\dd x', \dd y')\\
                &=\iint \|x' - y'\|^p \pi_{x,y}(\dd x', \dd y')\gamma(\dd x, \dd y)\\
                &=  \int \W_p^p(K(x, \bullet), K(y,\bullet))\gamma(\dd x, \dd y).
            \end{align*}
        \end{enumerate}
    \end{proof}

    \begin{proposition}[Contractions]\label{prop:wass_contract}
        Let $K$ be a Markov kernel and $K(x, \bullet)\in \W_p(\R^d)~\forall x\in \R^d$. Then
        \[
            \tau_p(K) = \sup_{\mu,\nu\in\calP_p(\R^d),\mu\neq \nu}\frac{\W_p(\mu K, \nu K)}{\W_p(\mu,\nu)} = \sup_{x\neq y}\frac{\W_p(K(x, \bullet), K(y, \bullet))}{\|x -y\|}.
        \]
    \end{proposition}
    \begin{proof}
        For one inequality, we have
        \[
            \sup_{x \neq y}\frac{\W_p(K(x, \bullet),K(y, \bullet))}{\|x - y\|}=\sup_{\delta_x \neq \delta_y}\frac{\W_p(\delta_xK,\delta_yK)}{\W_p(\delta_x, \delta_y)}\leq \sup_{\mu\neq \nu}\frac{\W_p(\mu K,\nu K)}{\W_p(\mu,\nu)}
        \]For the reverse inequality
        \begin{align*}
            \W^p_p(\mu K, \nu K)&\leq \inf_{\gamma\in \calC(\mu,\nu)}\int \W_p^p(K(x, \bullet), K(y, \bullet))\gamma(\dd x, \dd y)\\
            &= \inf_{\gamma\in \calC(\mu,\nu)}\int \frac{\W^p_p(K(x, \bullet), K(y, \bullet))}{\|x-y\|^p}\|x-y\|^p\gamma(\dd x, \dd y)\\
            &\leq \sup_{x\neq y}\frac{\W^p_p(K(x, \bullet), K(y, \bullet))}{\|x-y\|^p} \inf_{\gamma\in \calP(\mu,\nu)}\int \|x-y\|^p \gamma(\dd x, \dd y)\\
            &= \sup_{x\neq y}\frac{\W^p_p(K(x, \bullet), K(y, \bullet))}{\|x-y\|^p} \W_p^p(\mu,\nu)
        \end{align*}
        hence dividing and raising both sides to the power $1/p$ implies
        \[
            \sup_{\mu\neq\nu}\frac{\W_p(\mu K,\nu K)}{\W_p(\mu,\nu)} \leq \sup_{x\neq y}\frac{\W_p( K(x, \bullet), K(y,\bullet))}{\|x - y\|}.
        \]
    \end{proof}

    \subsection{Tensorization}
    First we have the notation
    \begin{align*}
        [N]&:= \{1,\dots,N\}\\
        (N)_q&:= \frac{N!}{(N-q)!}\\
        [N]^{[q]}&:= \{\alpha :[q]\to [N]\}\\
        [N,q]&:=\{\alpha:[q]\to[N]~|~ \alpha\text{ is 1-1}\}
    \end{align*}

    Let $X:=(X^1,\dots,X^N)$ and define
    \[
        m(X)^{\otimes q}:= \frac{1}{N^q}\sum_{\alpha\in [N]^{[q]}}\delta_{(X^{\alpha(1)},\dots,X^{\alpha(q)})}
    \]and
    \[
        m(X)^{\odot q}:= \frac{1}{(N)_q}\sum_{\alpha\in [N, q]}\delta_{(X^{\alpha(1)},\dots,X^{\alpha(q)})}.
    \]Note that $|[N]^{[q]}|=N^q$ and $|[N,q]|=(N)_q$, and that $m(X)^{\otimes q}$ is precisely equal to the actual tensor product of $m(X)$ $q$-times since $\delta_{X^{\alpha(1)}}\cdots \delta_{X^{\alpha(q)}}=\delta_{(X^{\alpha(1)},\dots,X^{\alpha(q)})}$. There is a useful connection between tensor products and expectations: Let $\eta^{N,q}:=\law(X^1,\dots,X^q)$ where $X^i$ are \emph{exchangeable} but not necessarily independent. Then by definition we can write $\eta^{N,q}(f) = \E[f(X^1,\dots,X^q)]$. We have the equality
    \[
        \eta^{N,q}(f) = \E[f(X^1,\dots,X^q)] = \E[m(X)^{\odot q}].
    \]To see this, the $X^i$ are exchangeable so each q-subset of distinct elements has the same law. Hence $\E[f(X^{\alpha(1)},\dots,X^{\alpha(q)})]=\E[f(X^1,\dots,X^q)]$ for each $\alpha\in [N,q]$. The result immediately follows.
    
    \TensorizationLemma*

    \begin{proof}
        \begin{align*}
            \W_1(\mu^{N,q},\nu^{\otimes q})&\leq \E[\W_1((m(X))^{\odot q},\nu^{\otimes q})]\\
            &\leq \E[\W_1(m(X)^{\odot q},m(X)^{\otimes q})] + \E[\W_1(m(X)^{\otimes q},\nu^{\otimes q})]\\
            &\leq \E[\W_1(m(X)^{\odot q},m(X)^{\otimes q})] + q \E[\W_1(m(X),\nu)].
        \end{align*}
        The first term is handled in Lemma~\ref{lem:odot}.
    \end{proof}

    \begin{lemma}\label{lem:odot}
        Let $X:=(X^1,\dots,X^N)$ be a collection of exchangeable random variables. Then
        \[
            \E[\W_1(m(X)^{\otimes q}, m(X)^{\odot q})]\leq 2\frac{q^3}{N}\E[\|X^1\|]
        \]
    \end{lemma}
    \begin{proof}
        We know that the following "combinatorial transport" equation holds from \cite[Proposition~8.6.1]{del2004feynman} 
        \[
            m(X)^{\otimes q} = m(X)^{\odot q} R^{(q)}_N
        \]
        where 
        \[
            R^{(q)}_N: = \frac{(N)_q}{N^q}Id + (1 - \frac{(N)_q}{N^q})\tilde{R}^{(q)}_N
        \] 
        where is a $\tilde{R}^{(q)}_N$ is a Markov kernel. Then
        \begin{align*}
            \W_1(m(X)^{\odot q} R^{(q)}_N, m(X)^{\odot q}) &= W_1(m(X)^{\odot q} R^{(q)}_N, m(X)^{\odot q}Id)\\
            &\leq (1-\frac{(N)_q}{N^q})\W_1(m(X)^{\odot q}\tilde{R}^N_q, m(X)^{\odot q})\\
            &\leq \frac{q^2}{N}\W_1(m(X)^{\odot q}\tilde{R}^N_q, m(X)^{\odot q})
        \end{align*}
        But we have
        \[
            \W_1(m(X)^{\odot q}\tilde{R}^N_q, m(X)^{\odot q})\leq \M_1(m(X)^{\odot q}\tilde{R}^N_q) + \M_1(m(X)^{\odot q}).
        \]Also, taking expectations
        \begin{align*}
            \E[\M_1(m(X)^{\odot q})]&=\E\left[\frac{1}{(N)_q}\sum_{\alpha\in [N,q]}\|X^{\alpha(1)},\dots,X^{\alpha(q)}\|\right]\\
            &= \frac{1}{(N)_q}\sum_{\alpha\in [N,q]}\E[\|(X^{\alpha(1)},\dots,X^{\alpha(q)})\|]\\
            &= \E[\|(X^1,\dots,X^q)\|.
        \end{align*}
        But $\wtilde{R}^N_q$ is a purely combinatorial construction that treats the indices completely symmetrically, and the $X^i$s are exchangeable, hence under expectation we must also have this symmetry and so
        \[
            \E[\M_1(m(X)^{\odot q}\tilde{R}^N_q)] = \E[\|(X^1,\dots,X^q)\|].
        \]The proof is completed by noting that $\|x\|\leq \|x^1\| + \cdots + \|x^q\|$ and using exchangeability.
    \end{proof}

\end{document}